\documentclass[11pt]{article}

\usepackage{amsmath}
\usepackage{amsfonts}
\usepackage{amssymb}
\usepackage{wasysym}
\usepackage{amsthm}
\usepackage{appendix}
\usepackage{xcolor}
\usepackage[margin=1in]{geometry}

    \theoremstyle{definition}\newtheorem{rema}{Remark}[section]
    \theoremstyle{plain}
    \newtheorem{theo}[rema]{Theorem}
    \newtheorem{defi}[rema]{Definition}
    \newtheorem{lemma}[rema]{Lemma}
    \newtheorem{corol}[rema]{Corollary}
    \newtheorem{proposition}[rema]{Proposition}
    \theoremstyle{definition}

\newcommand{\qpej}[2]{(q^{2#1};q^2)_{#2}}
\newcommand{\qpoj}[2]{(-q^{2#1};q^2)_{#2}}
\newcommand{\ignore}[1]{}
\newcommand{\nh}[3]{{}_#1h^{\left( #3 \right)}_{#2}}
\newcommand{\rnh}[4]{{}_#1^#4h^{\left( #3 \right)}_{#2}}
\newcommand{\rnht}[4]{{}_#1^#4{\Tilde{h}}^{\left( #3 \right)}_{#2}}
\numberwithin{equation}{section}

\allowdisplaybreaks

\begin{document}
\title{Ghost series and a motivated proof of the Bressoud-G{\"o}llnitz-Gordon identities
}
\author{John Layne, Samuel Marshall, Christopher Sadowski, Emily Shambaugh}
\date{}
\bibliographystyle{alpha}
\maketitle

\begin{abstract}
We present what we call a ``motivated proof" of the Bressoud-G{\"o}llnitz-Gordon partition identities. Similar ``motivated proofs" have been given by Andrews and Baxter for the Rogers-Ramanujan identities and by Lepowsky and Zhu for Gordon's identities. Additionally, ``motivated proofs" have also been given for the Andrews-Bressoud partition identities by Kanade, Lepowsky, Russell, and Sills and for the G{\"o}llnitz-Gordon-Andrews identities by Coulson, Kanade, Lepowsky, McRae, Qi, Russell, and the third author. Our proof borrows both the use of ``ghost series" from the ``motivated proof" of the Andrews-Bressoud identities and uses recursions similar to those found in the ``motivated proof" of the G{\"o}llnitz-Gordon-Andrews identities. We anticipate that this ``motivated proof" of the Bressoud-G{\"o}llnitz-Gordon identities will illuminate certain twisted vertex-algebraic constructions.
\end{abstract}

\section{Introduction}

The Rogers-Ramanujan partition identities are the pair of partition identities which we write as
\begin{equation}\label{RR1}
    \prod_{n \ge 1, n \not\equiv  0,\pm 2 \bmod 5} \frac{1}{1-q^n} = \sum_{m \ge 0}p_1(m)q^m
\end{equation}
and
\begin{equation}\label{RR2}
    \prod_{n \ge 1, n \not\equiv  0,\pm 1 \bmod 5} \frac{1}{1-q^n} = \sum_{m \ge 0}p_2(m)q^m,
\end{equation}
where $p_1(m)$ enumerates the number of partitions of $m$ where adjacent parts have difference at least $2$ and $p_2(m)$ enumerates the number of partitions of $m$ where adjacent parts have difference at least $2$ and in which no $1$ appears. We shall refer to the right-hand side of identities such as (\ref{RR1}) and (\ref{RR2}) as ``sum sides." We note that in (\ref{RR1}) and (\ref{RR2}), the left-hand sides count the number of partitions into parts $\pm 1 \bmod 5$ and the number of partitions into parts $\pm 2 \bmod 5$, respectively. We refer to left-hand sides such as these as ``product sides." In \cite{AB}, related to work by Baxter in \cite{B}, Andrews and Baxter gave what they called a ``motivated proof" of the Rogers-Ramanujan identities. The initial motivation was a question posed by Leon Ehrenpreis: By definition, it is clear that $p_1(m) \ge p_2(m)$ for all $m \ge 0$. Can one see this fact from the product sides of the identities without knowledge of the Rogers-Ramanujan identities? That is, without knowledge of (\ref{RR1}) and (\ref{RR2}), and starting only with the product sides of these identities, can one show that 
\begin{equation}
    \prod_{n \ge 1, n \not\equiv 0, \pm 2 \bmod 5}\frac{1}{1-q^n} - \prod_{n \ge 1, n \not\equiv 0, \pm 1 \bmod 5}\frac{1}{1-q^n} = \sum_{m \ge 0} a_m q^m,
\end{equation}
where $a_m \ge 0$ for all $m \ge 0$? In \cite{AB}, Andrews and Baxter answer this question in the affirmative and, in doing so, also prove the Rogers-Ramanujan identities.

The product sides of many partition identities arise naturally in the representation theory of vertex operator algebras (as in \cite{LM}). It is therefore natural to ask whether one can exhibit sum sides to these identities using vertex-algebraic techniques. In \cite{LW1}-\cite{LW4}, Lepowsky and Wilson exhibited these sum sides using monomials in principally twisted $Z$-operators. In their work, they gave a vertex-algebraic interpretation of the Gordon-Andrews-Bressoud identities and a purely vertex-algebraic proof of the Rogers-Ramanujan identities. Their proof, however, does not compare different product sides by taking various linear combinations of their product sides, but rather, considers the relevant representations (and thus the relevant products) individually. A related proof of the Gordon-Andrews-Bressoud identities was given in \cite{MP}. We refer the reader to the survey articles \cite{L} and \cite{K} for further reading about these approaches and recent advances. We also refer the reader to \cite{CMPP}, where new partition identities arising from representations of the affine Lie algebra $C_{\ell}^{(1)}$ have been found (these identities have been proved for the level $1$ case in \cite{DK} and \cite{Ru}). Another approach for obtaining such a sum side is to construct exact sequences among relevant representations of a given vertex operator algebra (for examples, see \cite{CLM1}-\cite{CLM2}, \cite{CalLM}, \cite{PSW}, \cite{BKRS}, and many others). Namely, such an exact sequence gives a recursion satisfied by the graded dimensions of various representations, whose solution readily leads to a sum side. In the context of exact sequences among representations, expressions in which the product sides are subtracted from one another with division by pure powers of the formal variable $q$ have a natural interpretation in terms of maps among the graded spaces of these representations. In particular, it is expected that the ideas in the motivated proofs in \cite{AB}, \cite{LZ}, \cite{CKLMQRS}, \cite{KLRS}, and the present work will lead to a ``categorification" of these identities in terms of certain twisted representations for certain generalized vertex operator algebras and exact sequences among these representations involving twisted intertwining operators (see \cite{DL}, \cite{Hu}). The program of finding ``motivated proofs" of partition identities is thus motivated by the desire to better understand these underlying vertex-algebraic structures. Such a program is underway. Thus, we shall use ``motivated proof" as a technical term and drop the quotations. For a much more detailed exposition of these ideas, we refer the reader to the introductions of \cite{LZ}, \cite{CKLMQRS}, and \cite{KLRS}. 

We now give a brief explanation of Andrews's and Baxter's motivated proof, which is related to Baxter's proof in \cite{B} and Rogers' and Ramanujan's proof in \cite{RR}. Let $G_1$ and $G_2$ denote the left-hand sides of (\ref{RR1}) and (\ref{RR2}). Empirically, Andrews and Baxter noticed that $G_3:=\frac{G_1-G_2}{q}$ is an element of $1 + q^3\mathbb{C}[[q]]$. More generally, they observed and proved that \begin{equation}\label{ABEmpHyp}
G_{j+2}:=\frac{G_{j} - G_{j+1}}{q^j} \in 1 + q^{j+2} \mathbb{C}[[q]] \end{equation}
for $j \ge 1$ and called (\ref{ABEmpHyp}) the Empirical Hypothesis. One important consequence to the Empirical Hypothesis, which is key in Andrews's and Baxter's proof, is the fact that $$\lim_{j \to \infty}G_j = 1$$
(here we say the limit exists if the coefficient of each power of $q$ stabilizes as $j \to \infty$). To answer Ehrenpreis's question, using the recursive definition of $G_j$, Andrews and Baxter expressed $G_3$  as
\begin{equation}
    G_3 := A_{j}G_{j} + B_{j} G_{j + 1},
\end{equation}
where $j \ge 3$ and $A_{j},B_{j} \in \mathbb{C}[q]$. They then give combinatorial interpretations for $A_j$ and $B_j$ and showed that $\lim_{j \to \infty}B_j= 0$ and $A_\infty:= \lim_{j \to \infty}A_j$ exists and counts partitions of positive integers into parts which differ by at least $2$ and in which $1$ and $2$ do not appear. This, in turn, gives
\begin{equation}\label{G3lim}
    G_3 = \lim_{j \to \infty}\left(A_jG_j + B_jG_{j+1}\right) = \lim_{j \to \infty}\left(A_j\right)\lim_{j \to \infty}\left(G_j\right) + \lim_{j \to \infty}\left(B_j\right) \lim_{j \to \infty}\left(G_{j+1}\right) = A_\infty,
\end{equation}
thus answering Ehrenpreis's question. Repeating this procedure for $G_1$ and $G_2$ leads to a proof of the Rogers-Ramanujan identities. It is now natural to ask: given the product sides of a family of partition identities, and without knowledge of the sum sides of these identities, can one deduce the sum sides of these identities using similar techniques?

Andrews's and Baxter's motivated proof has since been extended to several other partition identities. In general, much of the structure of these proofs is similar. Of particular importance in these proofs are recursions generalizing and extending the recursive definition in (\ref{ABEmpHyp}). We briefly review the recursive definitions which extend (\ref{ABEmpHyp}). In \cite{LZ}, Lepowsky and Zhu generalized Andrews's and Baxter's proof to Gordon's identities (cf. \cite{G}), the odd-modulus generalization of the Rogers-Ramanujan identities. In their proof, Lepowsky and Zhu introduced what they call ``shelves." In particular, let $G_i$ for $1 \le i \le k$ be the product sides of Gordon's identities $\bmod 2k+1$ (when $k=2$, these are just the product sides of (\ref{RR1}) and (\ref{RR2})). These expressions live on what is called ``shelf $0$." Next, given a series on ``shelf $j$," denoted by $G_{(k-1)j+i}$ for some $j\ge 0$ and $1 \le i \le k$, Lepowsky and Zhu then define the series on ``shelf $j+1$" by the tautological expression $G_{(k-1)(j+1)+1} = G_{(k-1)j+k}$ and by
\begin{equation}\label{LZRecursions}
    G_{(k-1)(j+1) + i}:=\frac{G_{(k-1)j + k-i+1} - G_{(k-1)j + k-i+2}}{q^{j(i-1)}} 
\end{equation}
for $2 \le i \le k$. The Empirical Hypothesis in \cite{LZ} states that 
\begin{equation}\label{LZEH}
    G_{(k-1)j + i}\in \begin{cases}  1 + q^{j+1}\mathbb{C}[[q]] & \text{if } 1 \le i \le k-1\\
     1 + q^{j+2}\mathbb{C}[[q]] & \text{if } i=k.
    \end{cases}
\end{equation} 
In \cite{CKLMQRS}, a motivated proof of the G{\"o}llnitz-Gordon-Andrews identity is given (cf. \cite{Gol}, \cite{A2}, and Chapter 7 of \cite{A4}). As in \cite{LZ}, the authors define $G_i$ for $1 \le i \le k$ to be the product sides of the G{\"o}llnitz-Gordon-Andrews identities $\bmod \ 4k$. For these identities, a more complicated-looking recursion is needed in place of (\ref{LZRecursions}). In particular, the series on shelf $j+1$ for $j \ge 0$ are defined by 
\begin{equation}
    G_{(k-1)(j+1)+1} = G_{(k-1)j+k}
\end{equation}
and 
\begin{equation}\label{CKLMQRSRecursions}
    G_{(k-1)(j+1) + i}:=\frac{G_{(k-1)j + k-i+1} - G_{(k-1)j + k-i+2}}{q^{2j(i-1)}} -q^{-1}G_{(k-1)(j+1)+i-1}
\end{equation}
for $2\leq i\leq k$.
A key difference between (\ref{LZRecursions}) and (\ref{CKLMQRSRecursions}) is the use of the $(i-1)^{\mathrm{st}}$ series on shelf $j+1$ to define the $i^{\mathrm{th}}$ series on shelf $j+1$. Lastly, in \cite{KLRS}, a motivated proof of the Andrews-Bressoud $\bmod\  2k$ partition identities is given. As before, let $G_i$ for $1 \le i \le k$ denote the product sides of the Andrews-Bressoud identities $\bmod \ 2k$ (in \cite{KLRS}, they are denoted $B_i$). The recursion needed to define the higher shelves in \cite{KLRS} is quite different than (\ref{LZRecursions}) and (\ref{CKLMQRSRecursions}). In particular, due to parity conditions arising in the sum side of these identities (see condition (2) in the introduction to \cite{KLRS}), division by a pure power of $q$ is replaced by division by a sum of two powers of $q$. This type of division has no obvious interpretation in the vertex-algebraic framework discussed above and is thus not ``motivated." In order to define the series $G_{(k-1)(j+1)+i}$ on the higher shelves using division by pure powers of $q$, the authors introduce what they call ``ghost series." They denote these series by $\tilde{G}_{(k-1)(j+1)+i}$ for $2 \le i\le k$ on each shelf and define the higher shelves as follows: 
\begin{equation}
    G_{(k-1)(j+1) + 1} = G_{(k-1)j+k},
\end{equation}
\begin{equation}\label{KLRSGhost1}
    G_{(k-1)(j+1)+2} = \frac{G_{(k-1)j+k-1} - \tilde{G}_{(k-1)j+k}}{q^{j+1}} = \tilde{G}_{(k-1)j+k},
\end{equation}
and
\begin{equation}\label{KLRSGhost2}
    G_{(k-1)(j+1)+i} = \frac{G_{(k-1)j+k-i+1}-\tilde{G}_{(k-1)j+k-i+2}}{q^{(j+1)(i-1)}} = \frac{\tilde{G}_{(k-1)j+k-i+2} - G_{(k-1)j+k-i+3}}{q^{(j+1)(i-2)}}.
\end{equation}
We note here that (\ref{KLRSGhost1}) and (\ref{KLRSGhost2}) also serve as definitions for the ghost series. In all of these works, once the series on the higher shelves have been established and a general formula for them has been proved, it is straightforward to show that a variant of the Empirical Hypothesis (\ref{LZEH}) holds.  

We now discuss the identities that are the focus of the present work and compare the approach we use to the approaches in \cite{LZ}, \cite{CKLMQRS}, and \cite{KLRS}. The Bressoud-G{\"o}llnitz-Gordon identities, proved by Bressoud (\cite{Br}), are an extension of the G{\"o}llnitz-Gordon-Andrews identities to moduli of the form $4k-2$. In particular, we use the statement of these identities as presented in Corollary 1.3 of \cite{CoLoMa}, with $i$ replaced by $k-i+1$ in the statement of the theorem. For $2 \le i \le k$, these identities state
\begin{equation}\label{BGG}
    \prod_{n \ge 0} \frac{(1+q^{2n+1})(1-q^{2k-2i+1 + (4k-2)n})(1-q^{2k+2i-3+(4k-2)n})(1-q^{(4k-2)(n+1)})}{1-q^{2n+2}} = \sum_{n \ge 0} a_{i}(n)q^n,
\end{equation}
where $a_i(n)$ enumerates partitions $\lambda$ of $n$ such that:
\begin{enumerate}
    \item Each odd part appears at most once,
    \item $f_1(\lambda) + f_2(\lambda) \le k-i$,
    \item $f_{2t}(\lambda) + f_{2t+1}(\lambda) + f_{2t+2}(\lambda) \le k-1$ for all $t\geq 0$,
    \item If $f_{2t}(\lambda) + f_{2t+1}(\lambda) + f_{2t+2}(\lambda) = k-1$, then
    \begin{equation}\label{parity}
        tf_{2t}(\lambda) + (t+1)(f_{2t+1}(\lambda) + f_{2t+2}(\lambda)) \equiv k-i+V^o_\lambda(t) \bmod 2,
    \end{equation}
\end{enumerate}
where $f_t(\lambda)$ denotes the number of occurrences of $t$ in $\lambda$ and $V^o_\lambda(t)$ denotes the number of odd parts in $\lambda$ which do not exceed $2t$. We note here that the left-hand side of (\ref{BGG}) enumerates partitions where:
\begin{enumerate}
    \item Even parts are multiples of $4$ not divisible by $8k-4$,
    \item Odd parts are not congruent to $\pm(2k-2i+1) \bmod 4k-2$ with parts congruent to $2k-1 \bmod 4k-2$ appearing at most once.
\end{enumerate}
(For some interesting recent work related to these identities, we refer the reader to \cite{HWZ}, \cite{HJZ1}, \cite{HJZ2}, and \cite{HZ}.)

In the motivated proof in the present work, we define our series $G_i$ for $1 \le i \le k$ on shelf $j=0$ using the left-hand side of (\ref{BGG}), extending the definition to $i=1$. As in \cite{KLRS}, condition 4 prevents us from defining the higher shelves using solely the series $G_i$ and division by pure powers of $q$.  In order to define the series on the higher shelves, we use a mix of ideas from \cite{CKLMQRS} and \cite{KLRS}. In particular, for $j \ge 0$, we have, tautologically, that $G_{(k-1)(j+1)+1} = G_{(k-1)j+k}$, and we introduce the series $\tilde{G}_{(k-1)j+i}$ for $2 \le i \le k$, which we call ``ghost series." We define both the ghost series and the series on shelf $j+1$ by the recursions
\begin{equation}
    G_{(k-1)(j+1)+2} = \frac{G_{(k-1)j+k-1}-\tilde{G}_{(k-1)j + k}}{q^{2(j+1)}} - q^{-1}G_{(k-1)(j+1)+1} = \tilde{G}_{(k-1)j+k}
\end{equation}
and
\begin{align}
    G_{(k-1)(j+1)+i} &= \frac{G_{(k-1)j + k-i+1}- \tilde{G}_{(k-1)j+k-i+2}}{q^{2(j+1)(i-1)}} - q^{-1}G_{(k-1)(j+1) + i-1}\\
    &= \frac{\tilde{G}_{(k-1)j+k-i+2} - G_{(k-1)j+k-i+3}}{q^{2(j+1)(i-2)}} - q^{-1}G_{(k-1)(j+1)+i-1}
\end{align}
for $3 \le i \le k$. We note that these recursions use ideas from (\ref{CKLMQRSRecursions}) - (\ref{KLRSGhost2}). In particular, we will also show that the ghost series, $\tilde{G}_i$ for $2 \le i \le k$, are  generating functions which enumerate partitions satisfying the same conditions as $a_i(n)$ but the right-hand side of (\ref{parity}) is replaced by $ k-i+1+V^o_\lambda(t) \bmod 2$.

We also note one more important difference between the proof in the present work and the proofs in \cite{AB}, \cite{LZ}, \cite{CKLMQRS}, and \cite{KLRS}. In these works, the recursions used to define the higher shelves are ``reversed" in order to write the series $G_i$ as polynomial linear combinations of series on higher shelves:
\begin{equation}
G_{i}=\nh{i}{1}{j}G_{(k-1)j+1} + \nh{i}{2}{j}G_{(k-1)j+2} + \dots + \nh{i}{k}{j}G_{(k-1)j+k},
\end{equation}
where the polynomials $\nh{i}{\ell}{j}$, for $1 \le \ell \le k$, are generalizations of $A_j$ and $B_j$ above. The authors then take the limit as $j \to \infty$ to obtain an argument similar to (\ref{G3lim}):
\begin{align*}
    G_{i}&=\lim_{j \to \infty}\left(\nh{i}{1}{j}G_{(k-1)j+1} + \nh{i}{2}{j}G_{(k-1)j+2} + \dots + \nh{i}{k}{j}G_{(k-1)j+k}\right)\\
    &= \left(\lim_{j \to \infty} \nh{i}{1}{j}\right)\left( \lim_{j \to \infty} G_{(k-1)j+1}\right)+ \dots + \left(\lim_{j \to \infty} \nh{i}{k}{j}\right)\left( \lim_{j \to \infty} G_{(k-1)j+k}\right)\\
    &= \nh{i}{1}{\infty} \cdot 1 + 0 \cdot 1 + \dots + 0 \cdot 1\\
    &=\nh{i}{1}{\infty}
\end{align*}
to complete their motivated proof, where it is clear from a matrix interpretation of their recursions that each of the limits, $\lim_{j \to \infty} \nh{i}{\ell}{j}$, exist for $1 \le \ell \le k$.  In the present work, it may not be the case that $\lim_{j \to \infty} \nh{i}{\ell}{j}$ in general exists, so we need to consider a more intricate argument. Indeed, we show that 
$\lim_{j \to \infty} \left(\nh{i}{1}{j} + \nh{i}{2}{j}\right)$
exists and require a slightly more intricate argument to conclude that $$G_i = \lim_{j \to \infty} \left(\nh{i}{1}{j} + \nh{i}{2}{j}\right)$$ for $1 \le i \le k$.

The present work is structured as follows: In Section 2, we recall certain standard notation 
regarding $q$-series and introduce the  official series $G_{\ell}$ and the ghost series $\tilde{G}_{\ell}$ which are the main objects of our study. In Section 3, we derive closed-form expressions for $G_{\ell}$ and $\tilde{G}_{\ell}$. In Section 4, these closed-form expressions are used to prove an Empirical Hypothesis for both the official series and the ghost series. In Section 5, we use the recursive definitions of the official series and ghost series to provide a matrix interpretation and write our series $G_i$, for $1 \le i \le k$, as polynomial linear combinations of series on higher shelves. In Section 6, we provide combinatorial interpretations for the polynomials from Section 5 and use them to complete our proof of the Bressoud-G{\"o}llnitz-Gordon identities. We also provide a combinatorial interpretation of the ghost series. In Section 7, we provide a dictionary between our closed-form expressions and specialization of the series $\tilde{J}(a;x;q)$ in \cite{CoLoMa}. We also give $(a;x;q)$-expressions governing the ghost series. In addition, we give combinatorial interpretations for these series and explore their properties.

\section*{Acknowledgement}
All four authors of this work are funded by NSF grant 1851948. This work was carried out as part of Ursinus College's NSF REU program during Summer 2023. We thank Matthew Russell for helpful comments on an early draft of this work and for helpful conversations.

\section{The formal series $G_{\ell}$ and $\tilde{G}_{\ell}$}

In this section, we establish notation that we will use throughout the paper. Throughout this work, let $k \ge 2$ be an integer and $a$, $x$, and $q$ be formal variables. All power series in this work are formal power series in $a$, $x$, and $q$.

Suppose $n$ is a nonnegative integer and $\lambda = (b_1,b_2,\dots,b_s)$ is a partition of $n$, with $b_1\geq\hdots\geq b_s$. For $t \in \mathbb{N}$, use $f_t(\lambda)$ to denote the number of occurrences of $t$ in $\lambda$ and $V^o_\lambda(t)$ to denote the number of odd parts in $\lambda$ which do not exceed $2t$. We also use the following standard $q$-Pochhammer notation throughout this work:
\begin{equation}
    (a;q)_n := (1-a)(1-aq)\cdots (1-aq^{n-1})
\end{equation}
and
\begin{equation}
    (a;q)_\infty:= \prod_{n \ge 1}(1-aq^{n-1}).
\end{equation}
We also will use
\begin{equation}
    (q)_n:=(q;q)_n
\end{equation}
and
\begin{equation}
 (q)_{\infty}:= (q;q)_\infty.
\end{equation}
Finally, we define
\begin{equation}
    (a_1,a_2,\dots, a_k;q)_n := (a_1;q)_n (a_2;q)_n \cdots (a_k;q)_n
\end{equation}
and similarly define
\begin{equation}
    (a_1,a_2,\dots, a_k;q)_\infty := (a_1;q)_\infty (a_2;q)_\infty \cdots (a_k;q)_\infty.
\end{equation}

Our main object of study will be the formal power series we denote by $G_{\ell}$ for integers $\ell \ge 1$. When $2 \le \ell \le k$, the $G_{\ell}$ denote the product sides of the Bressound-G{\"o}llnitz-Gordon identities. In particular, for $1 \le i \le k$, we define the series $G_i$ as presented in Corollary 1.3 of \cite{CoLoMa}. In particular, let
\begin{equation}
     G_i:= \frac{(-q;q^2)_{\infty}(q^{2k-2i+1},q^{2k+2i-3},q^{4k-2};q^{4k-2})_{\infty}}{(q^2;q^2)_{\infty}}.
\end{equation}
When $2 \le i \le k$, this is the generating function for partitions satisfying the conditions:
\begin{enumerate}
    \item Even parts are multiples of $4$ not divisible by $8k-4$,
    \item Odd parts are not congruent to $\pm(2k-2i+1) \bmod 4k-2$ with parts congruent to $2k-1 \bmod 4k-2$ appearing at most once.
\end{enumerate}
Note that we have replaced $i$ with $k-i+1$ and have also defined $G_i$ for $i=1$, as this will be necessary in our work.
We note that
\begin{equation}
    \frac{(-q;q^2)_\infty}{(q^2;q^2)_\infty} = \prod_{m \not\equiv 2 \bmod 4}\frac{1}{1-q^m},
\end{equation}
and write
\begin{equation}
    F(q) = \prod_{m \not\equiv 2 \bmod 4}(1-q^m).
\end{equation}
We note that, using Ramanujan's notation, we may write $F(q)$ as $\psi(-q)$, where
$$\psi(q) = f(q,q^3)$$ and $$f(a,b) = (-a,-b,ab;ab)_\infty.$$
We recall the Jacobi Triple Product identity
\begin{equation}\label{JTPIfirst}
    \sum_{n \in \mathbb{Z}} (-1)^nz^nq^{n^2} = (q^2,zq,z^{-1}q;q^2)_{\infty}.
\end{equation}
We rewrite the left-hand side of (\ref{JTPIfirst}) as
\begin{equation}
    \sum_{n \in \mathbb{Z}}(-1)^n z^n q^{n^2} = \sum_{n \ge 0} (-1)^n z^n q^{n^2}\left(1-z^{-2n-1}q^{2n+1}\right)
\end{equation}
so that our Jacobi Triple Product Identity is
\begin{equation}\label{JTPI}
\sum_{n \ge 0} (-1)^n z^n q^{n^2}\left(1-z^{-2n-1}q^{2n+1}\right) = (q^2,zq,z^{-1}q;q^2)_{\infty}.
\end{equation}
Making the substitution $q \rightarrow q^{2k-1}$ and $z \rightarrow q^{2i-2}$,  (\ref{JTPI}) becomes
\begin{equation}\label{JTPIrewrite}
    \sum_{n \geq 0} (-1)^n q^{(4k-2)\binom{n}{2} +(2i+2k-3)n} \left(1-q^{(2k-2i+1)(2n+1)}\right) =  (q^{4k-2},q^{2k-2i+1},q^{2k+2i-3};q^{4k-2})_{\infty}.
\end{equation}
Now, using (\ref{JTPIrewrite}), we write
\begin{equation}\label{shelf0}
    G_i = \frac{1}{F(q)}  \sum_{n \ge 0} (-1)^n q^{(4k-2)\binom{n}{2} +(2i+2k-3)n} \left(1-q^{(2k-2i+1)(2n+1)}\right)
\end{equation}
for $1 \le i \le k$.

As in \cite{KLRS}, we define what we call ghost series $\tilde{G}_i$, where $2\leq i\leq k$, as well as the series $G_{k+h}$ for $1\leq h\leq k-1$. In particular, define
\begin{equation}\label{ghostshelf0-1def}
    G_{k+1} = \frac{G_{k-1}-\tilde{G}_k}{q^2} - q^{-1}G_k = \tilde{G}_k
\end{equation}
and
\begin{equation}\label{ghostshelf0-2def}
    G_{k+h} = \frac{G_{k-h}-\tilde{G}_{k-h+1}}{q^{2h}}-q^{-1}G_{k+h-1}
    =\frac{\tilde{G}_{k-h+1} - G_{k-h+2}}{q^{2(h-1)}} - q^{-1}G_{k+h-1}
\end{equation}
for $2\leq h \leq k-1.$
\begin{rema}
    As in \cite{KLRS}, we have not defined the series $\tilde{G}_{1}$ as it is not necessary in our proofs. We will, however, attach a combinatorial meaning to this series in Remark \ref{Ghost1} and explore this series in more generality in Section 7.
\end{rema}
We note that (\ref{ghostshelf0-1def}) and (\ref{ghostshelf0-2def}) give us
\begin{equation}\label{ghostshelf0}
    \tilde{G}_{i} = \frac{G_{i-1} + q^{2}G_{i+1}}{1+q^2}
\end{equation}
for $2\leq i \leq k-1$ and
\begin{equation}\label{ghostshelf1}
    \tilde{G}_k = \frac{G_{k-1}-qG_k}{1+q^2}.
\end{equation}
\begin{rema}
    Later in our work, we will prove that the ghost series, $\tilde{G}_i$ for $2 \le i \le k$, enumerate partitions $\lambda$ which satisfy all the same conditions as those enumerated by $G_{i}$ for $2 \le i \le k$ with a change in a parity condition. Namely, we will show that $G_i$ enumerates partitions $\lambda$ satisfying:
    \begin{enumerate}
    \item Each odd part appears at most once,
    \item $f_1(\lambda) + f_2(\lambda) \le k-i$,
    \item $f_{2t}(\lambda) + f_{2t+1}(\lambda) + f_{2t+2}(\lambda) \le k-1$ for all $t\geq 0$,
    \item If $f_{2t}(\lambda) + f_{2t+1}(\lambda) + f_{2t+2}(\lambda) = k-1$, then
    \begin{equation}\label{GiList}
        tf_{2t}(\lambda) + (t+1)(f_{2t+1}(\lambda) + f_{2t+2}(\lambda)) \equiv k-i+V^o_\lambda(t) \bmod 2,
    \end{equation}
\end{enumerate}
and $\tilde{G}_i$ enumerates partitions satisfying all the same conditions as $G_i$ with (\ref{GiList}) replaced by 
\begin{equation}
    tf_{2t}(\lambda) + (t+1)(f_{2t+1}(\lambda) + f_{2t+2}(\lambda)) \not\equiv k-i+V^o_\lambda(t) \bmod 2.
\end{equation}
\end{rema}

Next, we use (\ref{shelf0}), (\ref{ghostshelf0}), and (\ref{ghostshelf1}) to derive closed-form expressions for $\tilde{G}_i$, where ${2\le i\le k}$. 
\begin{proposition}\label{shelf0ghosts}
    For $2\leq i \leq k$, we have
    \begin{align*}
        &\tilde{G}_i = \frac{1}{F(q)}\frac{1}{1+q^{2}}\sum_{n \ge 0} (-1)^nq^{(4k-2)\binom{n}{2} + (2k+2i-5)n}(1+q^{2(2n+1)})(1-q^{(2k-2i+1)(2n+1)}).
    \end{align*}
\end{proposition}
\begin{proof}
    First, for $2\leq i\leq k-1$, we have that 
    \begin{flalign*}
        &G_{i-1} + q^2G_{i+1} &\\
        &=\frac{1}{F(q)} \sum_{n \ge 0} (-1)^nq^{(4k-2)\binom{n}{2} + (2k+2i-5)n}(1-q^{(2k-2i+3)(2n+1)}) \\
        &\hphantom{=}+ \frac{1}{F(q)}\sum_{n \ge 0}(-1)^nq^{(4k-2)\binom{n}{2} + (2k+2i-1)n+2}(1-q^{(2k-2i-1)(2n+1)}) \\
        &= \frac{1}{F(q)} \sum_{n \ge 0 }(-1)^nq^{(4k-2)\binom{n}{2} + (2k+2i-5)n} \left( 1-q^{(2k-2i+3)(2n+1)} + q^{2(2n+1)}(1-q^{(2k-2i-1)(2n+1)})\right)\\
        &= \frac{1}{F(q)} \sum_{n \ge 0 }(-1)^nq^{(4k-2)\binom{n}{2} + (2k+2i-5)n}(1+q^{2(2n+1)})(1-q^{(2k-2i+1)(2n+1)}),
    \end{flalign*}
    which gives our result for $2\leq i\leq k-1$. We also have that 
    \begin{flalign}
        &G_{k-1} - qG_k &\\
        &= \frac{1}{F(q)}\sum_{n \ge 0} (-1)^n q^{(4k-2)\binom{n}{2} + (4k-5)n}(1-q^{3(2n+1)})\\
        &\hphantom{=}- \frac{1}{F(q)}\sum_{n \ge 0} (-1)^n q^{(4k-2)\binom{n}{2} + (4k-3)n+1}(1-q^{2n+1}) \\
        &=\frac{1}{F(q)} \sum_{n \ge 0} (-1)^n q^{(4k-2)\binom{n}{2} + (4k-5)n}\left(1-q^{3(2n+1)} - q^{2n+1}(1-q^{2n+1})\right)\\
        &=\frac{1}{F(q)} \sum_{n \ge 0} (-1)^n q^{(4k-2)\binom{n}{2} + (4k-5)n}(1+q^{2(2n+1)})(1-q^{2n+1}),
    \end{flalign}
    which gives the desired result for $\tilde{G}_k$.
\end{proof}

Note that we defined the series $G_1,\dots, G_k$ and $\tilde{G_2},\dots, \tilde{G}_k$, which, as in \cite{LZ}, \cite{KLRS}, and \cite{CKLMQRS}, we say are on shelf $j=0$. We now recursively define the series $G_{(k-1)j + i}$ for $ j \ge 1$ and $1 \le i \le k$. 
As in \cite{KLRS}, we also recursively define $\tilde{G}_{(k-1)j+i}$ for $j \ge 0$ and $2 \le i \le k$. Following \cite{KLRS} and \cite{CKLMQRS}, for $j\geq 0$, we define 
\begin{equation}\label{Gdef1}
    G_{(k-1)(j+1)+1} = G_{(k-1)j+k},
\end{equation}
\begin{equation}\label{Gdef2}
    G_{(k-1)(j+1)+2} = \frac{G_{(k-1)j+k-1}-\tilde{G}_{(k-1)j + k}}{q^{2(j+1)}} - q^{-1}G_{(k-1)(j+1)+1} = \tilde{G}_{(k-1)j+k},
\end{equation}
and
\begin{align}
    \label{Gdef3}
    G_{(k-1)(j+1)+i} &= \frac{G_{(k-1)j + k-i+1}- \tilde{G}_{(k-1)j+k-i+2}}{q^{2(j+1)(i-1)}} - q^{-1}G_{(k-1)(j+1) + i-1}\\
    \label{Gdef4}
    &= \frac{\tilde{G}_{(k-1)j+k-i+2} - G_{(k-1)j+k-i+3}}{q^{2(j+1)(i-2)}} - q^{-1}G_{(k-1)(j+1)+i-1}
\end{align}
for $3 \le i \le k$.
\begin{rema}
    The definition given by (\ref{Gdef1}) is immediate from the fact that 
    \begin{equation}
        (k-1)(j+1) + 1 = (k-1)j+k.
    \end{equation}
    We make note of this here, however, since in the next section we will give different expressions for $G_{(k-1)(j+1)+1}$ and $G_{(k-1)j + k}$ and will need to show that they are, in fact, equal. As in \cite{LZ}, \cite{KLRS}, and \cite{CKLMQRS}, we will call this ``edge-matching."
\end{rema}

We note here that, using (\ref{Gdef1}) - (\ref{Gdef4}), for $j\geq 0$, the ghosts may be explicitly defined by
\begin{equation}\label{GhostRec1} 
    \tilde{G}_{(k-1)j+i} = \frac{G_{(k-1)j+i-1} + q^{2(j+1)}G_{(k-1)j + i+1}}{1+q^{2(j+1)}}
\end{equation}
for $2\leq i\leq k-1$ and
\begin{equation}\label{GhostRec2}
    \tilde{G}_{(k-1)j+k} = \frac{G_{(k-1)j+k-1} - q^{2j+1}G_{(k-1)j+k}}{1+q^{2(j+1)}}.
\end{equation}

\section{Closed form of $G_{\ell}$ and $\tilde{G_{\ell}}$}
In this section, we provide a closed-form expression for the official series $G_{\ell}$ and the ghost series $\tilde{G}_{\ell}$.

\begin{theo}\label{ClosedFormG}
    For $j \ge 0$ and $1 \le i \le k$, we have that 
    \begin{equation}\label{GeneralG}
        G_{(k-1)j+i} \in \mathbb{C}[[q]]
    \end{equation}
    and, in fact,
    \begin{align}\label{Gs}
        &G_{(k-1)j+i} \nonumber\\
        &=\frac{1}{F(q)}\sum_{n \ge 0}\frac{(-1)^n q^{(4k-2)\binom{n}{2}+n(2k(j+1)+2(i-j)-3)}(-q^{2j+2};q^2)_n(q^{2(n+1)};q^2)_j}{(-q^2;q^2)_n(-q^{2n+1};q^2)_{j+1}}\\
        &\hspace{3.5em}\cdot\left(1-q^{2(2n+j+1)(k-i+1)} + q^{2n+2j+1}(1-q^{2(2n+j+1)(k-i)})\right).\nonumber
    \end{align}
Moreover, for $j \ge 0$ and $2 \le i \le k$, we have
\begin{equation}
    \tilde{G}_{(k-1)j+i} \in \mathbb{C}[[q]]
\end{equation}
and, in fact,
\begin{align}\label{gGs}
        &\tilde{G}_{(k-1)j+i} \nonumber\\
        &=\frac{1}{F(q)}\frac{1}{1+q^{2j+2}}\sum_{n \ge 0}\frac{(-1)^n q^{(4k-2)\binom{n}{2}+n(2k(j+1)+2(i-j-1)-3)}(-q^{2j+2};q^2)_n\qpej{(n+1)}{j}}{(-q^2;q^2)_n\qpoj{n+1}{j+1}}\\
     &\hspace{7.5em}\cdot(1+q^{2(2n+j+1)})\left(1-q^{2(2n+j+1)(k-i+1)} + q^{2n+2j+1}(1-q^{2(2n+j+1)(k-i)})\right).\nonumber
    \end{align}
\end{theo}

\begin{proof}
Throughout our proof, we will use the notation $\overline{G}_{j,i}$ for the right-hand side of (\ref{Gs}) and $\overline{\tilde{G}}_{j,i}$ for the right-hand side of (\ref{gGs}). We will first show that $G_{(k-1)j+i} = \overline{G}_{j,i}$ for $j \ge 0$ and $1 \le i \le k$ by induction on $j$ and $i$.
We note that the $j=0$ case is given by (\ref{shelf0}) for $G_i$ with $1 \le i \le k$ and that Proposition \ref{shelf0ghosts} gives us the $j=0$ case of (\ref{gGs}) for $\tilde{G}_i$ with $2 \le i\le k$.
    
Suppose, for some $j\geq 0$, that $G_{(k-1)j+i}=\overline{G}_{j,i}$ for $1\leq i\leq k$ and that $\Tilde{G}_{(k-1)j+i}=\overline{\Tilde{G}}_{j,i}$ for $2\leq i\leq k$. We first show that $G_{(k-1)(j+1)+1}=\overline{G}_{j+1,1}$. Since $G_{(k-1)j+k} = G_{(k-1)(j+1)+1}$, we must show that ``edge-matching" holds, i.e. that $G_{(k-1)j+k} = \overline{G}_{j+1,1}$ Indeed, we have
\begin{align*}
        &G_{(k-1)j+k}\\
        &=\frac{1}{F(q)}\sum_{n \geq 0}\frac{(-1)^n q^{(4k-2)\binom{n}{2}+n(2k(j+1)+2(k-j)-3)}(-q^{2j+2};q^2)_n\qpej{(n+1)}{j}}{(-q^2;q^2)_n\qpoj{n+1}{j+1}}\\
            &\hspace{3.5em}\cdot\left(1-q^{2(2n+j+1)(k-k+1)} + q^{2n+2j+1}(1-q^{2(2n+j+1)(k-k)})\right) \\
        &=\frac{1}{F(q)}\sum_{n \geq 0}\frac{(-1)^n q^{(4k-2)\binom{n}{2}+n(2k(j+2)-2j-3)}(-q^{2j+2};q^2)_n\qpej{(n+1)}{j}}{(-q^2;q^2)_n\qpoj{n+1}{j+1}}\\
            &\hspace{3.5em}\cdot(1-q^{2(2n+j+1)}) \\
        &=\frac{1}{F(q)}\sum_{n \geq 0}\frac{(-1)^n q^{(4k-2)\binom{n}{2}+n(2k(j+2)-2j-3)}(-q^{2j+2};q^2)_n\qpej{(n+1)}{j}}{(-q^2;q^2)_n\qpoj{n+1}{j+1}}\\
            &\hspace{3.5em}\cdot(1-q^{2(2n+j+1)})\frac{1+q^{2j+2}}{1+q^{2j+2}} \\
        &=\frac{1}{F(q)}\sum_{n \geq 0}\frac{(-1)^n q^{(4k-2)\binom{n}{2}+n(2k(j+2)-2j-3)}(-q^{2j+2};q^2)_n\qpej{(n+1)}{j}}{(-q^2;q^2)_n\qpoj{n+1}{j+1}}\\
            &\hspace{3.5em}\cdot\frac{1+q^{2j+2}-q^{2(2n+j+1)}-q^{4(n+j+1)}}{1+q^{2j+2}} \\
        &=\frac{1}{F(q)}\sum_{n \geq 0}\frac{(-1)^n q^{(4k-2)\binom{n}{2}+n(2k(j+2)-2j-3)}(-q^{2j+2};q^2)_n\qpej{(n+1)}{j}}{(-q^2;q^2)_n\qpoj{n+1}{j+1}}\\
            &\hspace{3.5em}\cdot\frac{1-q^{4(n+j+1)}+q^{2j+2}(1-q^{4n})}{1+q^{2j+2}} \\
        &=\frac{1}{F(q)}\sum_{n \geq 0}\frac{(-1)^n q^{(4k-2)\binom{n}{2}+n(2k(j+2)-2j-3)}(-q^{2j+2};q^2)_n\qpej{(n+1)}{j}(1-q^{2(n+j+1)})}{(-q^2;q^2)_n\qpoj{n+1}{j+1}}\\
            &\hspace{3.5em}\cdot\frac{1+q^{2(n+j+1)}}{1+q^{2j+2}}\\
        &\hphantom{=}+\frac{1}{F(q)}\sum_{n \geq 1}\frac{(-1)^n q^{(4k-2)\binom{n}{2}+n(2k(j+2)-2j-3)}(-q^{2j+2};q^2)_n\qpej{(n+1)}{j}(1-q^{4n})}{(-q^2;q^2)_n\qpoj{n+1}{j+1}}\\
            &\hspace{4em}\cdot\frac{q^{2j+2}}{1+q^{2j+2}}\\
        &=\frac{1}{F(q)}\sum_{n \geq 0}\frac{(-1)^n q^{(4k-2)\binom{n}{2}+n(2k(j+2)-2j-3)}(-q^{2j+2};q^2)_n\qpej{(n+1)}{j+1}}{(-q^2;q^2)_n\qpoj{n+1}{j+1}}\\
            &\hspace{3.5em}\cdot\frac{1+q^{2(n+j+1)}}{1+q^{2j+2}}\\
        &\hphantom{=}-\frac{1}{F(q)}\sum_{n \geq 0}\frac{(-1)^n q^{(4k-2)\binom{n}{2}+n(2k(j+2)-2j-3)}(-q^{2j+2};q^2)_{n+1}\qpej{(n+2)}{j}(1-q^{4(n+1)})}{(-q^2;q^2)_{n+1}\qpoj{(n+1)+1}{j+1}}\\
            &\hspace{4em}\cdot\frac{q^{2j+2}}{1+q^{2j+2}}\cdot q^{(4k-2)n+2k(j+2)-2j-3}\\
        &=\frac{1}{F(q)}\sum_{n \geq 0}\frac{(-1)^n q^{(4k-2)\binom{n}{2}+n(2k(j+2)-2j-3)}(-q^{2j+2};q^2)_n\qpej{(n+1)}{j+1}}{(-q^2;q^2)_n\qpoj{n+1}{j+1}}\\
            &\hspace{3.5em}\cdot\frac{1+q^{2(n+j+1)}}{1+q^{2j+2}}\\
        &\hphantom{=}-\frac{1}{F(q)}\sum_{n \geq 0}\frac{(-1)^n q^{(4k-2)\binom{n}{2}+n(2k(j+2)-2j-3)}(-q^{2(j+1)+2};q^2)_{n}\qpej{(n+2)}{j}(1-q^{4(n+1)})}{(-q^2;q^2)_{n+1}\qpoj{(n+1)+1}{j+1}}\\
            &\hspace{4em}\cdot q^{(4k-2)n+2k(j+2)-2j-3}q^{2j+2}\\
        &=\frac{1}{F(q)}\sum_{n \geq 0}\frac{(-1)^n q^{(4k-2)\binom{n}{2}+n(2k(j+2)-2j-3)}(-q^{2j+2};q^2)_n\qpej{(n+1)}{j+1}}{(-q^2;q^2)_n\qpoj{n+1}{j+1}}\\
            &\hspace{3.5em}\cdot\frac{1+q^{2(n+j+1)}}{1+q^{2j+2}}\\
        &\hphantom{=}-\frac{1}{F(q)}\sum_{n \geq 0}\frac{(-1)^n q^{(4k-2)\binom{n}{2}+n(2k(j+2)-2j-3)}(-q^{2(j+1)+2};q^2)_{n}\qpej{(n+1)}{j+1}}{(-q^2;q^2)_{n}\qpoj{(n+1)+1}{j+1}}\\
            &\hspace{4em}\cdot q^{(4k-2)n+2k(j+2)-1}\\
        &=\frac{1}{F(q)}\sum_{n \geq 0}\frac{(-1)^n q^{(4k-2)\binom{n}{2}+n(2k(j+2)-2j-3)}(-q^{2(j+1)+2};q^2)_n\qpej{(n+1)}{j+1}}{(-q^2;q^2)_n\qpoj{n+1}{j+1}}\\
        &\hphantom{=}-\frac{1}{F(q)}\sum_{n \geq 0}\frac{(-1)^n q^{(4k-2)\binom{n}{2}+n(2k(j+2)-2j-3)}(-q^{2(j+1)+2};q^2)_{n}\qpej{(n+1)}{j+1}}{(-q^2;q^2)_{n}\qpoj{(n+1)+1}{j+1}}\\
            &\hspace{4em}\cdot q^{(4k-2)n+2k(j+2)-1}\\
        &=\frac{1}{F(q)}\sum_{n \geq 0}\frac{(-1)^n q^{(4k-2)\binom{n}{2}+n(2k(j+2)-2j-3)}(-q^{2(j+1)+2};q^2)_n\qpej{(n+1)}{j+1}}{(-q^2;q^2)_n\qpoj{n+1}{j+2}}\\
            &\hspace{3.5em}\cdot(1+q^{2(n+j+1)+1})\\
        &\hphantom{=}-\frac{1}{F(q)}\sum_{n \geq 0}\frac{(-1)^n q^{(4k-2)\binom{n}{2}+n(2k(j+2)-2j-3)}(-q^{2(j+1)+2};q^2)_{n}\qpej{(n+1)}{j+1}}{(-q^2;q^2)_{n}\qpoj{n+1}{j+2}}\\
            &\hspace{4em}\cdot q^{(4k-2)n+2k(j+2)-1}(1+q^{2n+1})\\
        &=\frac{1}{F(q)}\sum_{n \geq 0}\frac{(-1)^n q^{(4k-2)\binom{n}{2}+n(2k(j+2)-2j-3)}(-q^{2(j+1)+2};q^2)_n\qpej{(n+1)}{j+1}}{(-q^2;q^2)_n\qpoj{n+1}{j+2}}\\
            &\hspace{3.5em}\cdot \left(1+q^{2(n+j+1)+1}-q^{(4k-2)n+2k(j+2)-1}(1+q^{2n+1})\right)\\
        &=\frac{1}{F(q)}\sum_{n \geq 0}\frac{(-1)^n q^{(4k-2)\binom{n}{2}+n(2k(j+2)-2j-3)}(-q^{2(j+1)+2};q^2)_n\qpej{(n+1)}{j+1}}{(-q^2;q^2)_n\qpoj{n+1}{j+2}}\\
            &\hspace{3.5em}\cdot \left(1+q^{2(n+j+1)+1}-q^{(4k-2)n+2k(j+2)-1}-q^{(4k-2)n+2k(j+2)-1+2n+1}\right)\\
        &=\frac{1}{F(q)}\sum_{n \geq 0}\frac{(-1)^n q^{(4k-2)\binom{n}{2}+n(2k(j+2)-2j-3)}(-q^{2(j+1)+2};q^2)_n\qpej{(n+1)}{j+1}}{(-q^2;q^2)_n\qpoj{n+1}{j+2}}\\
            &\hspace{3.5em}\cdot \left(1+q^{2(n+j+1)+1}-q^{-2n+2k(2n+j+2)-1}-q^{2k(2n+j+2)} \right)\\
        &=\frac{1}{F(q)}\sum_{n \ge 0}\frac{(-1)^n q^{(4k-2)\binom{n}{2}+n(2k(j+2)-2j-3)}(-q^{2(j+1)+2};q^2)_n\qpej{(n+1)}{j+1}}{(-q^2;q^2)_n\qpoj{n+1}{j+2}}\\
            &\hspace{3.5em}\cdot \left(1-q^{2k(2n+j+2)} + q^{2(n+j+1)+1}(1-q^{2(2n+j+2)(k-1)})\right)\\
        &=\overline{G}_{j+1,1}.
\end{align*}

Now, suppose
$G_{(k-1)(j+1)+s}=\overline{G}_{j+1,s}$ for all $1\leq s\leq i-1$, where $i$ satisfies $1\leq i-1\leq k-1$. We will show that
\begin{align*}
    \overline{G}_{(k-1)(j+1)+i}+ q^{-1}G_{(k-1)(j+1) + i-1}&= \frac{G_{(k-1)j + k-i+1}- \tilde{G}_{(k-1)j+k-i+2}}{q^{2(j+1)(i-1)}}.
\end{align*}
We have
\begin{align}\label{geeandghostv1}
        &\frac{G_{(k-1)j + k-i+1}- \tilde{G}_{(k-1)j+k-i+2}}{q^{2(j+1)(i-1)}} \nonumber \\
        &= \frac{1}{q^{2(j+1)(i-1)}} \frac{1}{F(q)}\sum_{n \ge 0}\frac{(-1)^n q^{(4k-2)\binom{n}{2}+n(2k(j+2)-2(i+j)-1)}(-q^{2j+2};q^2)_n\qpej{(n+1)}{j}}{(-q^2;q^2)_n\qpoj{n+1}{j+1}} \nonumber \\     
            &\hspace{8.5em}\cdot \left( 1-q^{2i(2n+j+1)} + q^{2n+2j+1}(1-q^{2(i-1)(2n+j+1)}) \right) \nonumber  \\
        &\hphantom{=}-\frac{1}{(1+q^{2j+2})q^{2(j+1)(i-1)}}\frac{1}{F(q)} \sum_{n \ge 0}\frac{(-1)^n q^{(4k-2)\binom{n}{2}+n(2k(j+2)- 2(i +j)-1)}(-q^{2j+2};q^2)_n\qpej{(n+1)}{j}}{(-q^2;q^2)_n\qpoj{n+1}{j+1}} \nonumber \\ 
            &\hspace{13em}\cdot(1+q^{2(2n+j+1)})\left(1-q^{2( i -1 )(2n+j+1)} + q^{2n+2j+1}(1-q^{2(i-2)(2n+j+1)})\right) \nonumber \\
        &= \frac{1}{(1+q^{2j+2})} \frac{1}{F(q)}\sum_{n \ge 0}\frac{(-1)^n q^{(4k-2)\binom{n}{2}+n(2k(j+2)-2(i+j)-1)}(-q^{2j+2};q^2)_n\qpej{(n+1)}{j}}{(-q^2;q^2)_n\qpoj{n+1}{j+1}} \nonumber \\     
            &\hspace{8.5em} \cdot q^{-2(j+1)(i-1)}\bigg( (1+q^{2j+2} )\left(1-q^{2i(2n+j+1)} + q^{2n+2j+1}(1-q^{2(i-1)(2n+j+1)}) \right) \nonumber \\
            &\hspace{9.5em} - (1+q^{2(2n+j+1)}) \left( 1-q^{2( i -1 )(2n+j+1)} + q^{2n+2j+1}(1-q^{2(i-2)(2n+j+1)}) \right) \bigg). 
\end{align}

Notice that the term in the last two lines of (\ref{geeandghostv1}) can be rewritten as
\begin{align*}
        & q^{-2(j+1)(i-1)}\bigg( (1+q^{2j+2} )\Big(1-q^{2i(2n+j+1)} + q^{2n+2j+1}(1-q^{2(i-1)(2n+j+1)})\Big) \nonumber \\
            &\hspace{1em} - (1+q^{2(2n+j+1)}) \Big( 1-q^{2( i -1 )(2n+j+1)} + q^{2n+2j+1}(1-q^{2(i-2)(2n+j+1)}) \Big)\bigg)\\
       =& q^{-2(j+1)(i-1)}\bigg( (1+q^{2j+2} )\Big(1 + q^{2n+2j+1}-q^{2(i-1)(2n+j+1)}(q^{2n+2j+1}+q^{2(2n+j+1)})\Big) \nonumber \\
            &\hspace{1em} - (1+q^{2(2n+j+1)}) \Big( 1 + q^{2n+2j+1}-q^{2(i-2)(2n+j+1)}(q^{2n+2j+1}+q^{2(2n+j+1)}) \Big)\bigg)\\
        =& q^{-2(j+1)(i-1)}\bigg( (1+q^{2j+2} )\Big(1 + q^{2n+2j+1}-q^{2(i-1)(2n+j+1)}(q^{2n+2j+1}+q^{2(2n+j+1)})\Big) \nonumber \\
            &\hspace{1em} - \Big( 1 + q^{2n+2j+1}-q^{2(i-2)(2n+j+1)}(q^{2n+2j+1}+q^{2(2n+j+1)}) \Big)\\
                &\hspace{2em}-\Big( q^{2(2n+j+1)}(1 + q^{2n+2j+1})-q^{2(i-1)(2n+j+1)}(q^{2n+2j+1}+q^{2(2n+j+1)}) \Big)\bigg)\\
        =& q^{-2(j+1)(i-1)}\bigg( (1+q^{2j+2} )\Big(1 + q^{2n+2j+1}-q^{2(i-1)(2n+j+1)}(q^{2n+2j+1}+q^{2(2n+j+1)})\Big) \nonumber \\
            &\hspace{1em} - \Big( 1 + q^{2n+2j+1}-q^{2(i-1)(2n+j+1)}(q^{2n+2j+1}+q^{2(2n+j+1)}) \Big)\\
                &\hspace{2em}-\Big( q^{2(2n+j+1)}(1 + q^{2n+2j+1})-q^{2(i-2)(2n+j+1)}(q^{2n+2j+1}+q^{2(2n+j+1)})\Big)\bigg)\\
        =& q^{-2(j+1)(i-1)}\bigg( q^{2j+2}\Big(1 + q^{2n+2j+1}-q^{2(i-1)(2n+j+1)}(q^{2n+2j+1}+q^{2(2n+j+1)})\Big) \nonumber \\
            &\hspace{1em}-\Big( q^{2(2n+j+1)}(1 + q^{2n+2j+1})-q^{2(i-2)(2n+j+1)}(q^{2n+2j+1}+q^{2(2n+j+1)})\Big)\bigg)\\
        =& q^{-2(j+1)(i-1)}\Big(q^{2j+2}(1+ q^{2n+2j+1})-q^{2(i-1)(2n+j+1)+2j+2}(q^{2n+2j+1}+q^{2(2n+j+1)}) \nonumber \\
            &\hspace{1em}- q^{2(2n+j+1)}(1+ q^{2n+2j+1})+q^{2(i-2)(2n+j+1)}(q^{2n+2j+1}+q^{2(2n+j+1)})\Big)\\
        =& q^{-2(j+1)(i-1)}\Big((q^{2j+2}- q^{2(2n+j+1)})(1+ q^{2n+2j+1})\\
            &\hspace{1em}+(q^{2(i-2)(2n+j+1)}-q^{2(i-1)(2n+j+1)+2j+2})(q^{2n+2j+1}+q^{2(2n+j+1)})\Big) \nonumber \\
        =& q^{-2(j+1)(i-1)}\Big(q^{2j+2}(1- q^{4n})(1+ q^{2n+2j+1})\\
            &\hspace{1em}+q^{2(i-2)(2n+j+1)+2n+2j+1}(1-q^{2(2n+j+1)+2j+2})(1+q^{2n+1})\Big) \nonumber \\
        =& q^{-2(j+1)(i-1)}\Big(q^{2j+2}(1- q^{4n})(1+ q^{2n+2j+1})\\
            &\hspace{1em}+q^{2(i-2)(2n+j+1)+2n+2j+1}(1-q^{4(n+j+1)})(1+q^{2n+1})\Big) \nonumber \\
        =& q^{-2(j+1)(i-2)}(1- q^{4n})(1+ q^{2(n+j)+1})\\
            &\hspace{1em}+q^{4(i-2)n+2n-1}(1-q^{4(n+j+1)})(1+q^{2n+1}) \nonumber\\
        =& q^{-2(i-2)(j+1)}(1- q^{4n})+q^{2n+2j+1-2(i-2)(j+1)}(1- q^{4n})\\
            &\hspace{1em}+q^{4(i-2)n+2n-1}(1-q^{4(n+j+1)})+q^{4(i-2)n+4n}(1-q^{4(n+j+1)}) \nonumber \\
        =& q^{-2(i-2)(j+1)}(1- q^{4n})+q^{2n-1-2(i-3)(j+1)}(1- q^{4n})\\
            &\hspace{1em}+q^{4(i-2)n+2n-1}(1-q^{4(n+j+1)})+q^{4(i-2)n+4n}(1-q^{4(n+j+1)}) \nonumber \\
        =& q^{-2(i-2)(j+1)}(1- q^{4n})+q^{2n-1-2(i-3)(j+1)}(1- q^{4n})\\
            &\hspace{1em}+q^{4(i-2)n+2n-1}(1-q^{4(n+j+1)})+q^{4(i-1)n}(1-q^{4(n+j+1)}) \nonumber \\
        =& q^{-2(i-2)(j+1)}(1- q^{4n})+q^{4(i-1)n}(1-q^{4(n+j+1)})\\
            &\hspace{1em}+q^{2n-1}\left(q^{-2(i-3)(j+1)}(1- q^{4n})+q^{4(i-2)n}(1-q^{4(n+j+1)})\right). \nonumber 
\end{align*}
So, we have
\begin{align}\label{splitSums}
        &\hspace{-8.5em}\frac{G_{(k-1)j + k-i+1}- \tilde{G}_{(k-1)j+k-i+2}}{q^{2(j+1)(i-1)}} \nonumber \\
        = \frac{1}{(1+q^{2j+2})} \frac{1}{F(q)}&\sum_{n \ge 0}\frac{(-1)^n q^{(4k-2)\binom{n}{2}+n(2k(j+2)-2(i+j)-1)}(-q^{2j+2};q^2)_n\qpej{(n+1)}{j}}{(-q^2;q^2)_n\qpoj{n+1}{j+1}} \nonumber \\
            &\cdot \Big(q^{-2(i-2)(j+1)}(1- q^{4n})+q^{4(i-1)n}(1-q^{4(n+j+1)}) \nonumber \\
                &\hspace{2em}+q^{2n-1}\big(q^{-2(i-3)(j+1)}(1- q^{4n})+q^{4(i-2)n}(1-q^{4(n+j+1)})\big)\Big) \nonumber \\
        = \frac{1}{(1+q^{2j+2})} \frac{1}{F(q)}&\sum_{n \ge 0}\frac{(-1)^n q^{(4k-2)\binom{n}{2}+n(2k(j+2)-2(i+j)-1)}(-q^{2j+2};q^2)_n\qpej{(n+1)}{j}}{(-q^2;q^2)_n\qpoj{n+1}{j+1}} \nonumber \\
            &\cdot (q^{-2(i-2)(j+1)}(1- q^{4n})+q^{4(i-1)n}(1-q^{4(n+j+1)}))\nonumber\\
        &\hspace{-7.5em} +\frac{1}{(1+q^{2j+2})} \frac{q^{-1}}{F(q)}\sum_{n \ge 0}\frac{(-1)^n q^{(4k-2)\binom{n}{2}+n(2k(j+2)-2(i+j-1)-1)}(-q^{2j+2};q^2)_n\qpej{(n+1}{j}}{(-q^2;q^2)_n\qpoj{n+1}{j+1}} \nonumber \\
            &\hspace{1em}\cdot (q^{-2(i-3)(j+1)}(1- q^{4n})+q^{4(i-2)n}(1-q^{4(n+j+1)})).
\end{align}
Since the second sum of (\ref{splitSums}) is the same as the first sum, except $i$ is replaced with $i-1$, we only need to consider the first sum, as a similar computation will hold for the second sum. Taking the first sum on the right-hand side of (\ref{splitSums}), we have
\begin{align*}
        &\frac{1}{(1+q^{2j+2})}\sum_{n \ge 0}\frac{(-1)^n q^{(4k-2)\binom{n}{2}+n(2k(j+2)-2(i+j)-1)}(-q^{2j+2};q^2)_n\qpej{(n+1)}{j}}{(-q^2;q^2)_n\qpoj{n+1}{j+1}} \nonumber \\
            &\hspace{5em}\cdot (q^{-2(i-2)(j+1)}(1- q^{4n})+q^{4(i-1)n}(1-q^{4(n+j+1)}))\\
        =&\sum_{n \ge 0}\frac{(-1)^n q^{(4k-2)\binom{n}{2}+n(2k(j+2)-2(i+j)-1)}(-q^{2j+2};q^2)_n\qpej{(n+1)}{j}(1-q^{4n})}{(1+q^{2j+2})(-q^2;q^2)_n\qpoj{n+1}{j+1}} \nonumber \\
            &\cdot q^{-2(i-2)(j+1)}\\
        \hphantom{=}&+\sum_{n \ge 0}\frac{(-1)^n q^{(4k-2)\binom{n}{2}+n(2k(j+2)-2(i+j)-1)}(-q^{2j+2};q^2)_n\qpej{(n+1)}{j}(1-q^{4(n+j+1)})}{(1+q^{2j+2})(-q^2;q^2)_n\qpoj{n+1}{j+1}} \nonumber \\
            &\hspace{1em}\cdot q^{4(i-1)n}\\
        =&\sum_{n \ge 1}\frac{(-1)^n q^{(4k-2)\binom{n}{2}+n(2k(j+2)+2(i-j)-5)}(-q^{2j+2};q^2)_n\qpej{(n+1)}{j}(1-q^{4n})}{(1+q^{2j+2})(-q^2;q^2)_n\qpoj{n+1}{j+1}} \nonumber \\
            &\cdot q^{-2(i-1)(2n+j+1)+2j+2}\\
        &+\sum_{n \ge 0}\frac{(-1)^n q^{(4k-2)\binom{n}{2}+n(2k(j+2)+2(i-j)-5)}(-q^{2(j+1)+2};q^2)_n\qpej{(n+1)}{j}(1-q^{2(n+j+1)})}{(1+q^{2j+2})(-q^2;q^2)_n\qpoj{n+1}{j+2}} \nonumber \\
            &\hspace{1em}\cdot (1+q^{2j+2})(1+q^{2(n+j+1)+1})\\
        =&-\sum_{n \ge 0}\frac{(-1)^n q^{(4k-2)\binom{n}{2}+n(2k(j+2)+2(i-j)-5)}(-q^{2(j+1)+2};q^2)_{n}\qpej{(n+1)}{j+1}}{(1+q^{2j+2})(-q^2;q^2)_{n}\qpoj{n+1}{j+2}} \nonumber \\
            &\hspace{1em}\cdot q^{-2(i-1)(2(n+1)+j+1)+2j+2+(4k-2)n+2k(j+2)+2(i-j)-5}(1+q^{2j+2})(1+q^{2n+1})\\
        &+\sum_{n \ge 0}\frac{(-1)^n q^{(4k-2)\binom{n}{2}+n(2k(j+2)+2(i-j)-5)}(-q^{2(j+1)+2};q^2)_n\qpej{(n+1)}{j+1}}{(1+q^{2j+2})(-q^2;q^2)_n\qpoj{n+1}{j+2}} \nonumber \\
            &\hspace{1em}\cdot (1+q^{2j+2})(1+q^{2(n+j+1)+1})\\
        =&\sum_{n \ge 0}\frac{(-1)^n q^{(4k-2)\binom{n}{2}+n(2k(j+2)+2(i-j)-5)}(-q^{2(j+1)+2};q^2)_n\qpej{(n+1)}{j+1}}{(-q^2;q^2)_n\qpoj{n+1}{j+2}} \nonumber \\
            &\cdot\left((1+q^{2(n+j+1)+1})-q^{2(k-i)(2n+j+2)+2n+2j+3}(1+q^{2n+1})\right)\\
        =&\sum_{n \ge 0}\frac{(-1)^n q^{(4k-2)\binom{n}{2}+n(2k(j+2)+2(i-j)-5)}(-q^{2(j+1)+2};q^2)_n\qpej{(n+1)}{j+1}}{(-q^2;q^2)_n\qpoj{n+1}{j+2}} \nonumber \\
            &\cdot\left(1-q^{2(k-i+1)(2n+j+2)}+q^{2(n+j+1)+1}(1-q^{2(k-i)(2n+j+2)})\right),
\end{align*}
which is exactly $\overline{G}_{j+1,i}$. Substituting $i$ for $i-1$ yields $G_{(k-1)(j+1)+i-1}$. Therefore,
\begin{align*}
    \overline{G}_{(k-1)(j+1)+i}+ q^{-1}G_{(k-1)(j+1) + i-1}&= \frac{G_{(k-1)j + k-i+1}- \tilde{G}_{(k-1)j+k-i+2}}{q^{2(j+1)(i-1)}}.
\end{align*}
So, $\overline{G}_{j,i}=G_{(k-1)j+i}$ for all $j\geq 0$ and $1\leq i\leq k$.

Lastly, we use our formula for $G_{(k-1)(j+1)+i}$ to prove (\ref{gGs}). 
For $2 \le i \le k-1$, we have  
\begin{align*}
        & \hspace{-8em} \tilde{G}_{(k-1)(j+1)+i}=\frac{G_{(k-1)(j+1)+i-1} + q^{2(j+2)}G_{(k-1)(j+1) + i+1}}{1+q^{2(j+2)}}\\
        =\frac{1}{F(q)}\frac{1}{1+q^{2(j+2)}}&\sum_{n \ge 0}\frac{(-1)^n q^{(4k-2)\binom{n}{2}+n(2k(j+2)+2(i-j-2)-3)}(-q^{2(j+1)+2};q^2)_n\qpej{(n+1)}{j+1}}{(-q^2;q^2)_n\qpoj{n+1}{j+2}} \nonumber \\
            &\cdot\Big(1-q^{2(2n+j+2)(k-i+2)} + q^{2n+2j+3}(1-q^{2(2n+j+2)(k-i+1)})\Big)\nonumber\\
        &\hspace{-7.5em}+\frac{1}{F(q)}\frac{q^{2(j+2)}}{1+q^{2(j+2)}}\sum_{n \ge 0}\frac{(-1)^n q^{(4k-2)\binom{n}{2}+n(2k(j+2)+2(i-j)-3)}(-q^{2(j+1)+2};q^2)_n\qpej{(n+1)}{j+1}}{(-q^2;q^2)_n\qpoj{n+1}{j+2}}\nonumber\\
            &\hspace{1em}\cdot\Big(1-q^{2(2n+j+2)(k-i)} + q^{2n+2j+3}(1-q^{2(2n+j+2)(k-i-1)})\Big)\nonumber\\
        =\frac{1}{F(q)}\frac{1}{1+q^{2(j+2)}}&\sum_{n \ge 0}\frac{(-1)^n q^{(4k-2)\binom{n}{2}+n(2k(j+2)+2(i-j-2)-3)}(-q^{2(j+1)+2};q^2)_n\qpej{(n+1)}{j+1}}{(-q^2;q^2)_n\qpoj{n+1}{j+2}} \nonumber \\
            &\cdot \bigg(\Big(1-q^{2(2n+j+2)(k-i+2)} + q^{2n+2j+3}(1-q^{2(2n+j+2)(k-i+1)})\Big)\nonumber\\
                &\hspace{2em}+ q^{2(2n+j+2)}\Big(1-q^{2(2n+j+2)(k-i)} + q^{2n+2j+3}(1-q^{2(2n+j+2)(k-i-1)})\Big) \bigg)\\
        =\frac{1}{F(q)}\frac{1}{1+q^{2(j+2)}}&\sum_{n \ge 0}\frac{(-1)^n q^{(4k-2)\binom{n}{2}+n(2k(j+2)+2(i-j-2)-3)}(-q^{2(j+1)+2};q^2)_n\qpej{(n+1)}{j+1}}{(-q^2;q^2)_n\qpoj{n+1}{j+2}} \nonumber \\
            &\cdot \bigg(\Big(1 + q^{2n+2j+3}-q^{2(2n+j+2)(k-i+1)}(q^{2n+2j+3}+q^{2(2n+j+2)})\Big)\nonumber\\
                &\hspace{2em}+ q^{2(2n+j+2)}\Big(1+ q^{2n+2j+3}-q^{2(2n+j+2)(k-i-1)}(q^{2n+2j+3}+q^{2(2n+j+2)} )\Big) \bigg)\\
        =\frac{1}{F(q)}\frac{1}{1+q^{2(j+2)}}&\sum_{n \ge 0}\frac{(-1)^n q^{(4k-2)\binom{n}{2}+n(2k(j+2)+2(i-j-2)-3)}(-q^{2(j+1)+2};q^2)_n\qpej{(n+1)}{j+1}}{(-q^2;q^2)_n\qpoj{n+1}{j+2}} \nonumber \\
            &\cdot \Big((1+q^{2(2n+j+2)})(1 + q^{2n+2j+3})\\
                &\hspace{2em}-(q^{2(2n+j+2)(k-i+1)}+q^{2(2n+j+2)(k-i)})(q^{2n+2j+3}+q^{2(2n+j+2)})\Big)\nonumber\\
        =\frac{1}{F(q)}\frac{1}{1+q^{2(j+2)}}&\sum_{n \ge 0}\frac{(-1)^n q^{(4k-2)\binom{n}{2}+n(2k(j+2)+2(i-j-2)-3)}(-q^{2(j+1)+2};q^2)_n \qpej{(n+1)}{j+1}}{(-q^2;q^2)_n\qpoj{n+1}{j+2}} \nonumber \\
            &\cdot \Big((1+q^{2(2n+j+2)})(1 + q^{2n+2j+3})\\
                &\hspace{2em}-q^{2(2n+j+2)(k-i)}(1+q^{2(2n+j+2)})(q^{2n+2j+3}+q^{2(2n+j+2)})\Big)\nonumber\\
        =\frac{1}{F(q)}\frac{1}{1+q^{2(j+2)}}&\sum_{n \ge 0}\frac{(-1)^n q^{(4k-2)\binom{n}{2}+n(2k(j+2)+2(i-j-2)-3)}(-q^{2(j+1)+2};q^2)_n\qpej{(n+1)}{j+1}}{(-q^2;q^2)_n\qpoj{n+1}{j+2}} \nonumber \\
            &\cdot (1+q^{2(2n+j+2)})\Big(1 + q^{2n+2j+3}-q^{2(2n+j+2)(k-i)}(q^{2n+2j+3}+q^{2(2n+j+2)})\Big)\nonumber\\
        =\frac{1}{F(q)}\frac{1}{1+q^{2(j+2)}}&\sum_{n \ge 0}\frac{(-1)^n q^{(4k-2)\binom{n}{2}+n(2k(j+2)+2(i-j-2)-3)}(-q^{2(j+1)+2};q^2)_n\qpej{(n+1)}{j+1}}{(-q^2;q^2)_n\qpoj{n+1}{j+2}}\nonumber\\
            &\cdot(1+q^{2(2n+j+2)})\Big(1-q^{2(2n+j+2)(k-i+1)} + q^{2n+2j+3}(1-q^{2(2n+j+2)(k-i)})\Big)\nonumber\\
        &\hspace{-8.2em}=\overline{\tilde{G}}_{j+1,i}
\end{align*}
and that
\begin{align*}
        &\hspace{-8em} \tilde{G}_{(k-1)(j+1)+k}=\frac{G_{(k-1)(j+1)+k-1} - q^{2(j+1)+1}G_{(k-1)(j+1)+k}}{1+q^{2(j+2)}}\\
        =\frac{1}{F(q)}\frac{1}{1+q^{2(j+2)}}&\sum_{n \ge 0}\frac{(-1)^n q^{(4k-2)\binom{n}{2}+n(2k(j+2)+2(k-j-2)-3)}(-q^{2(j+1)+2};q^2)_n\qpej{(n+1)}{j+1}}{(-q^2;q^2)_n\qpoj{n+1}{j+2}}\\
            &\cdot(1-q^{4(2n+j+2)} + q^{2(n+j+1)+1}(1-q^{2(2n+j+2)}) \\
        &\hspace{-7.5em}-\frac{1}{F(q)}\frac{q^{2(j+1)+1}}{1+q^{2(j+2)}}\sum_{n \ge 0}\frac{(-1)^n q^{(4k-2)\binom{n}{2}+n(2k(j+2)+2(k-j-1)-3)}(-q^{2(j+1)+2};q^2)_n\qpej{(n+1)}{j+1}}{(-q^2;q^2)_n\qpoj{n+1}{j+2}}\\
            &\hspace{1em}\cdot(1-q^{2(2n+j+2)}) \\
        =\frac{1}{F(q)}\frac{1}{1+q^{2(j+2)}}&\sum_{n \ge 0}\frac{(-1)^n q^{(4k-2)\binom{n}{2}+n(2k(j+2)+2(k-j-2)-3)}(-q^{2(j+1)+2};q^2)_n\qpej{(n+1)}{j+1}}{(-q^2;q^2)_n\qpoj{n+1}{j+2}}\\
            &\cdot \left( (1-q^{4(2n+j+2)} + q^{2(n+j+1)+1}(1-q^{2(2n+j+2)}) - q^{2(n+j+1)+1}(1-q^{2(2n+j+2)})\right)\\
        =\frac{1}{F(q)}\frac{1}{1+q^{2(j+2)}}&\sum_{n \ge 0}\frac{(-1)^n q^{(4k-2)\binom{n}{2}+n(2k(j+2)+2(k-j-2)-3)}(-q^{2(j+1)+2};q^2)_n\qpej{(n+1)}{j+1}}{(-q^2;q^2)_n\qpoj{n+1}{j+2}}\\
            &\cdot(1+q^{2(2n+j+2)})(1-q^{2(2n+j+2)})\\
        &\hspace{-8.2em}=\overline{\tilde{G}}_{j+1,k}, 
\end{align*}
thus proving (\ref{gGs}) for $j \ge 0$ and $2 \le i \le k$.
\end{proof}

\section{The Empirical Hypothesis}
We now formulate and prove the Empirical Hypothesis, which is a consequence of Theorem \ref{ClosedFormG}.
\begin{theo}\label{EmpiricalHypothesis}(Empirical Hypothesis)
    For all $j \ge 0$ and $1\leq i\leq k$, we have
    \begin{equation}
        G_{(k-1)j+i} = 1 + q^{2j+1}\gamma(q)
    \end{equation}
    for some $\gamma(q) \in \mathbb{C}[[q]]$.
\end{theo}
\begin{proof}
    This proof is similar to the proof of the Empirical Hypothesis (Theorem 4.1) in \cite{CKLMQRS}.
    We first note that when $n \ge 1$, we have
    \begin{align}
        (4k-2)\binom{n}{2} + (2k(j+1)+2(i-j)-3)n &\ge 2k(j+1)+2(i-j)-3\\
        &\ge 4(j+1)+2(i-j)-3\\
        &= 2j+2i+1\\
        &\ge 2j+3.
    \end{align}
    So, when $n\ge 1$, the powers of $q$ in (\ref{Gs}) are all at least $2j+3$. Thus, it suffices to examine the $n=0$ term in (\ref{Gs}). The $n=0$ term in (\ref{Gs}) is
    \begin{equation}
        \frac{(1-q^2)(1-q^4)\cdots (1-q^{2j})(1-q^{2(j+1)(k-i+1)}+q^{2j+1}(1-q^{2(j+1)(k-i)}))}{(1+q)(1+q^3)\cdots(1+q^{2j+1})\prod_{m \not\equiv 2 \bmod 4}(1-q^m)},
    \end{equation}
    which is identical to the $n=0$ term in Theorem 4.1 of \cite{CKLMQRS}, where it was proved that it is of the form 
    \begin{equation}
        1 + q^{2j+1}g(q)
    \end{equation}
    for some $g(q) \in \mathbb{C}[[q]]$. Thus, we have that 
    \begin{equation}
        G_{(k-1)j+i} = 1 + q^{2j+1}\gamma(q)
    \end{equation}
    for some $\gamma(q) \in \mathbb{C}[[q]]$.
\end{proof}

\begin{rema}
    We note that our use of the word ``empirical" in the Empirical Hypothesis is purely technical, and, unlike in \cite{AB}, our Empirical Hypothesis was not found empirically. In retrospect, by examining the combinatorial conditions in the sum sides of the Bressoud-G{\"o}llnitz-Gordon identities, one can obtain, by experimentation, the appropriate recursions defining shelf $j+1$ (which imply the Empirical Hypothesis) by taking appropriate linear combinations of elements on shelf $j$ and shelf $j+1$. 
\end{rema}

\begin{rema}
    Just as in \cite{CKLMQRS}, we have that 
    $$G_{(k-1)j + k} = G_{(k-1)(j+1)+1},$$
   and so, $G_{(k-1)j + k} = 1 + q^{2j+3}\gamma(q)$ for some $\gamma(q) \in \mathbb{C}[[q]].$
\end{rema}

We now also provide an Empirical Hypothesis for the ghost series.
\begin{theo}
    For all $j \ge 0$ and $2\leq i \leq k$, we have
    \begin{equation}
        \tilde{G}_{(k-1)j+i} = 1 + q^{2j+1}\gamma(q)
    \end{equation}
    for some $\gamma(q) \in \mathbb{C}[[q]]$.
\end{theo}
\begin{proof}
    We proceed as in the proof of Theorem \ref{EmpiricalHypothesis}. When $n \ge 1$, we have
    \begin{align}
        (4k-2)\binom{n}{2} + (2k(j+1)+2(i-j-1)-3)n &\ge 2k(j+1)+2(i-j-1)-3\\
        &\ge 4(j+1)+2(i-j-1)-3\\
        &= 2j+2i-1\\
        &\ge 2j+3,
    \end{align}
    where the last step follows from the fact that $2 \le i \le k$. So, when $n\ge 1$, the powers of $q$ in (\ref{Gs}) are all at least $2j+3$, and thus, we only need to consider the $n=0$ term. When $n=0$, we have 
    \begin{equation}\label{Ghostn0}
        \frac{(1-q^2)(1-q^4)\cdots (1-q^{2j})(1-q^{2(j+1)(k-i+1)}+q^{2j+1}(1-q^{2(j+1)(k-i)}))}{(1+q^{2j+2})(1+q)(1+q^3)\cdots(1+q^{2j+1})\prod_{m \not\equiv 2 \bmod 4}(1-q^m)}(1+q^{2(j+1)}),
    \end{equation}
    which is the same as the $n=0$ term of (\ref{Gs}). The result now follows.
\end{proof}

\section{Matrix interpretation and consequences}
Using (\ref{Gdef1}) - (\ref{Gdef4}), we have the following recursions satisfied by $G_{\ell}$ for $\ell \ge 1$:
\begin{equation}\label{Grecursion1}
    G_{(k-1)(j+1)+1} = G_{(k-1)j+k},
\end{equation}
\begin{equation}\label{Grecursion2}
    G_{(k-1)(j+1)+2} = \frac{G_{(k-1)j+k-1} - q^{2j+1}G_{(k-1)j+k}}{1+q^{2j+2}}, 
\end{equation}
and
\begin{equation}\label{Grecursion3}
    G_{(k-1)(j+1)+i} = \frac{G_{(k-1)j + k-i+1} - G_{(k-1)j+k-i+3}}{(1+q^{2j+2})q^{2(j+1)(i-2)}}-q^{-1}G_{(k-1)(j+1)+i-1}
\end{equation}
for $3\leq i\leq k$.
We note that, using (\ref{Grecursion1}), (\ref{Grecursion2}) can be rewritten as
\begin{equation}\label{Grecursion2rewrite}
    q^{2j+1}G_{(k-1)(j+1)+1} + (1+q^{2j+2})G_{(k-1)(j+1)+2} = G_{(k-1)j+k-1},
\end{equation}
and that (\ref{Grecursion3}) can be rewritten as
\begin{align}\label{Grecursion3rewrite}
    q^{-1}(1+q^{2j+2})&G_{(k-1)(j+1)+i-1} + (1+q^{2j+2})G_{(k-1)(j+1)+i}\nonumber \\
    &=q^{-2(j+1)(i-2)}G_{(k-1)j+k-i+1} - q^{-2(j+1)(i-2)}G_{(k-1)j+k-i+3}.
\end{align}

Define the vector
 \begin{equation*}
{\bf G}_{(0)} = 
\left[ \begin{array}{c}
G_{1} \\
\vdots \\
G_{k}
\end{array} \right],
\end{equation*}
and more generally, for each $j\geq 0$, define the vector
 \begin{equation*}
{\bf G}_{(j)} = 
\left[ \begin{array}{c}
G_{(k-1)j+1}\\
\vdots \\
G_{(k-1)j+k}
\end{array} \right].
\end{equation*}

For each $j \ge 1$, set
\begin{align}
{ \bf  B}_{(j)} =  
\left[\begin{matrix}
0 & 0 & 0 & \cdots & 0 & 0 & 0 & 1\\
0 & 0 & 0 & \cdots & 0 & 0 & 1 & 0\\
0 & 0 & 0 & \cdots & 0 & q^{-2j} & 0 & -q^{-2j}\\
0 & 0 & 0 & \cdots & q^{-4j} & 0 & -q^{-4j} & 0\\
\vdots & \vdots & \vdots & \swarrow & \swarrow & \swarrow & \vdots & \vdots\\
& & & & & & &\\
& & & & & & &\\
q^{-2j(k-2)} & 0 & -q^{-2j(k-2)} & \cdots & 0 & 0 & 0 & 0
\end{matrix}\right]\label{Bmatrix}
\end{align}
and set
\begin{equation*}
{\bf C}_{(j)} = 
\left[ \begin{array}{cccccc}
1           & 0                 & 0         & \cdots          & 0                & 0  \\
q^{2j-1}    & 1+q^{2j}          & 0         & \cdots          & 0                & 0  \\
0           & q^{-1}(1+q^{2j})  & 1+q^{2j}  & \cdots          & 0                & 0  \\
\vdots      & \vdots            & \vdots    & \searrow        & \vdots           & \vdots  \\
0           & 0                 & 0         & \cdots          & 1+q^{2j}         & 0  \\
0           & 0                 & 0         & \cdots          & q^{-1}(1+q^{2j}) & 1+q^{2j}
\end{array} \right] .
\end{equation*}
We now write (\ref{Grecursion1}), (\ref{Grecursion2rewrite}), and (\ref{Grecursion3rewrite}) as
\begin{equation}
    {\bf C}_{(j)} {\bf G}_{(j)} = {\bf B}_{(j)} {\bf G}_{(j-1)}
\end{equation}
for $j \ge 1$. 

Following \cite{CKLMQRS}, we define
\begin{equation}
    {\bf A}_{(j)} = {\bf B}_{(j)}^{-1}
\end{equation}
and note that when $k$ is even we have
\begin{align}
{ \bf  A}_{(j)} =  
\left[\begin{matrix}
    0       & 1         & 0         & q^{4j}    & \cdots    & q^{2j \left( k-4 \right) } & 0                         & q^{2j \left( k-2 \right) } \\
    1       & 0         & q^{2j}    & 0         & \cdots    & 0                          & q^{2j \left( k-3 \right)} & 0 \\
    0       & 1         & 0         & q^{4j}    & \cdots    & q^{2j \left( k-4 \right)}  & 0                         & 0 \\
    \vdots  & \vdots    & \vdots    & \vdots    &\swarrow   & \vdots                     & \vdots                    & \vdots  \\
    0       & 1         & 0         & q^{4j}    & \cdots    & 0                          & 0                         & 0 \\
    1       & 0         & q^{2j}    & 0         & \cdots    & 0                          & 0                         & 0 \\
    0       & 1         & 0         & 0         &\cdots     & 0                          & 0                         & 0 \\
    1       & 0         & 0         & 0         &\cdots     & 0                          & 0                         & 0 
\end{matrix}\right]\label{Amatrixeven}
\end{align}
and when $k$ is odd we have
\begin{align}
    { \bf  A}_{(j)} =  
    \left[\begin{matrix}
    1       & 0         & q^{2j} & 0      & \cdots    & q^{2j \left( k-4 \right)} & 0                           & q^{2j \left( k-2 \right) }\\
    0       & 1         & 0      & q^{4j} &\cdots     & 0                         & q^{2j \left( k -3 \right)}  & 0 \\
    1       & 0         & q^{2j} & 0      & \cdots    & q^{2j \left( k-4 \right)} & 0                           & 0 \\
    \vdots  & \vdots    & \vdots & \vdots &  \swarrow & \vdots                    & \vdots                      & \vdots\\
    0       & 1         & 0      & q^{4j} & \cdots    & 0                         & 0                           & 0 \\
    1       & 0         & q^{2j} & 0      & \cdots    & 0                         & 0                           & 0 \\
    0       & 1         & 0      & 0      & \cdots    & 0                         & 0                           & 0 \\
    1       & 0         & 0      &  0     &\cdots     & 0                         & 0                           & 0
\end{matrix}\right]\label{Amatrixodd}.
\end{align}
Next, we define
\begin{equation}
    {\bf A}'_{(j)} = {\bf A}_{(j)} {\bf C}_{(j)}.
\end{equation}
%
%
%
%
For $k$ odd we have

${\bf{A}}_{(j)}'=(1+q^{2j})\cdot$
\begin{equation}
 \left[\begin{matrix}
\left( 1 + q^{2j} \right)^{-1}              & q^{2j-1}  & q^{2j}        & q^{6j-1}      & \cdots     & q^{2j \left( k-4 \right) }       & q^{2j \left( k - 2 \right) - 1 }  & q^{2j \left( k-2 \right) } \\
q^{2j-1} \left( 1 + q^{2j} \right)^{-1}     & 1         & q^{4j - 1}    & q^{4j}        & \cdots     & q^{2j \left( k - 3 \right) - 1 } & q^{2j \left( k-3 \right)}       & 0\\
\left( 1 + q^{2j} \right)^{-1}              & q^{2j-1}  & q^{2j}        & q^{6j - 1}    & \cdots     & q^{2j \left( k-4 \right) }       & 0                                    & 0 \\
\vdots                                      & \vdots    & \vdots        & \vdots        & \swarrow   & \vdots                           & \vdots                               & \vdots \\
q^{2j - 1} \left( 1 + q^{2j} \right)^{-1}   & 1         & q^{4j - 1}    & q^{4j}        & \cdots     & 0                                & 0                                    & 0 \\
\left( 1 + q^{2j} \right)^{-1}              & q^{2j-1}  & q^{2j}        & 0             & \cdots     & 0                                & 0                                    & 0\\
q^{2j-1} \left( 1 + q^{2j} \right)^{-1}     & 1         & 0             & 0             & \cdots     & 0                                & 0                                    & 0 \\
\left( 1 + q^{2j} \right)^{-1}              & 0         & 0             & 0             & \cdots     & 0                                & 0                                    & 0 
\end{matrix}\right]
\end{equation}
and for $k$ even we have

${\bf{A}}_{(j)}'=(1+q^{2j})\cdot$
\begin{equation}
\left[\begin{matrix}
q^{2j-1} \left( 1 + q^{2j} \right)^{-1}     & 1             & q^{4j-1}      & q^{4j}        & \cdots    & q^{2j \left( k-4 \right) }        & q^{2j \left( k - 2 \right) - 1 }  & q^{2j \left( k-2 \right) } \\
\left( 1 + q^{2j} \right)^{-1}              & q^{2j - 1}    & q^{2j}        & q^{6j-1}      & \cdots    & q^{2j \left( k - 3 \right) - 1 }  & q^{2j \left( k-3 \right)}         & 0 \\
q^{2j - 1} \left( 1 + q^{2j} \right)^{-1}   & 1             & q^{4j - 1}    & q^{4j}        & \cdots    & q^{2j \left( k-4 \right) }        & 0                                  & 0 \\
\vdots                                      & \vdots        & \vdots        & \vdots        & \swarrow  & \vdots                            & \vdots & \vdots & \\
q^{2j - 1} \left( 1 + q^{2j} \right)^{-1}   & 1             & q^{4j - 1}    & q^{4j}        & \cdots    & 0                                 & 0                                  & 0 \\
\left( 1 + q^{2j} \right)^{-1}              & q^{2j-1}      & q^{2j}        & 0             & \cdots    & 0                                 & 0                                  & 0\\
q^{2j-1} \left( 1 + q^{2j} \right)^{-1}     & 1             & 0             & 0             &\cdots     & 0                                 & 0                                  & 0 \\
\left( 1 + q^{2j} \right)^{-1}              & 0             & 0             & 0             &\cdots     & 0                                 & 0                                  & 0 
\end{matrix}\right].
\end{equation}

In particular, we now have
\begin{equation}\label{VectorRecursion}
    {\bf G}_{(j-1)} = {\bf A}'_{(j)}{\bf G}_{(j)}
\end{equation}
for all $j \ge 1$.

Now, we fix an integer $J \ge 0$, which, as in \cite{CKLMQRS}, will denote a ``starting shelf." Now, if $j \ge J+1$, we repeatedly apply (\ref{VectorRecursion}) to obtain
\begin{equation}\label{JMatrixRecursion}
    {\bf G}_{(J)} = {\bf A}_{(J+1)}' {\bf A}_{(J+2)}' \cdots {\bf A}_{(j)}' {\bf G}_{(j)} =\ ^J{\bf h}^{(j)}{\bf G}_{(j)}, 
\end{equation}
where we define 
\begin{equation}\label{h-matrix-def}
    {}^J{\bf h}^{(j)} = {\bf A}_{(J+1)}' {\bf A}_{(J+2)}' \cdots {\bf A}_{(j)}'
\end{equation}
and take ${}^J{\bf h}^{(J)}$ to be the identity matrix. Writing out (\ref{JMatrixRecursion}) in component form, we have
\begin{equation}
    G_{(k-1)J+i} =  
    \rnh{i}{1}{j}{J} G_{(k-1) j +1} + \cdots + \rnh{i}{k}{j}{J} G_{(k-1) j + 
    k},
\end{equation}
where $\rnh{i}{\ell}{j}{J}$ is the $(i,\ell)$-entry of the matrix $^J{\bf h}^{(j)}$.

Now, using the definition of ${}^J{\bf h}^{(j)}$, we have for $j \ge J+1$ that
\begin{equation}
    {}^J{\bf h}^{(j)} = {}^J{\bf h}^{(j-1)}{\bf A}'_{(j)},
\end{equation}
which we now write component-wise. First, we consider the case when $k$ is even. When $\ell=1$ we have
\begin{align}\label{rec1}
    \rnh{i}{1}{j}{J} =& q^{2j-1}\left(\rnh{i}{1}{j-1}{J} + \rnh{i}{3}{j-1}{J} + \dots + \rnh{i}{k-1}{j-1}{J}\right)\nonumber \\
    &+ \left(\rnh{i}{2}{j-1}{J} + \rnh{i}{4}{j-1}{J} + \dots + \rnh{i}{k}{j-1}{J}\right),
\end{align}
when $\ell > 1$ is even we have
\begin{align}
    \rnh{i}{\ell}{j}{J} =& (1+q^{2j})q^{(\ell-2)(2j)}\left( \rnh{i}{1}{j-1}{J} + \rnh{i}{3}{j-1}{J} + \dots + \rnh{i}{k-\ell+1}{j-1}{J}\right) \nonumber \\
    &+(1+q^{2j})q^{(\ell-2)(2j)+2j-1}\left( \rnh{i}{2}{j-1}{J} + \rnh{i}{4}{j-1}{J} + \dots + \rnh{i}{k-\ell}{j-1}{J} \right),
\end{align}
and when $\ell > 1$ is odd we have
\begin{align}
    \rnh{i}{\ell}{j}{J} =& (1+q^{2j})q^{(\ell-2)(2j) + 2j-1}\left( \rnh{i}{1}{j-1}{J}+ \rnh{i}{3}{j-1}{J}) + \dots + 
    \rnh{i}{k-\ell}{j-1}{J}\right) \nonumber \\
    &+(1+q^{2j})q^{(\ell-2)(2j)}\left(\rnh{i}{2}{j-1}{J} + \rnh{i}{4}{j-1}{J} + \dots + \rnh{i}{k-\ell+1}{j-1}{J}\right).
\end{align}
Next, we consider the case when $k$ is odd. When $\ell=1$ we have
\begin{align}
    \rnh{i}{1}{j}{J} =& \left(\rnh{i}{1}{j-1}{J} + 
    \rnh{i}{3}{j-1}{J}+ \dots + 
    \rnh{i}{k}{j-1}{J}\right)\nonumber \\
    &+ q^{2j-1}\left(\rnh{i}{2}{j-1}{J} + 
    \rnh{i}{4}{j-1}{J} + \dots + \rnh{i}{k-1}{j-1}{J}\right),
\end{align}
when $\ell > 1$ is even we have
\begin{align}
    \rnh{i}{\ell}{j}{J}=& (1+q^{2j})q^{(\ell-2)(2j) + 2j-1}\left( \rnh{i}{1}{j-1}{J} + 
    \rnh{i}{3}{j-1}{J}+ \dots + \rnh{i}{k-\ell}{j-1}{J}\right)\nonumber  \\
    &+(1+q^{2j})q^{(\ell-2)(2j)}\left(\rnh{i}{2}{j-1}{J} + \rnh{i}{4}{j-1}{J} + \dots + \rnh{i}{k-\ell+1}{j-1}{J}\right),
\end{align}
and when $\ell > 1$ is odd we have
\begin{align}\label{rec2}
    \rnh{i}{\ell}{j}{J} =& (1+q^{2j})q^{(\ell-2)(2j) }\left( \rnh{i}{1}{j-1}{J} + \rnh{i}{3}{j-1}{J} + \dots + \rnh{i}{k-\ell+1}{j-1}{J}\right) \nonumber \\
    &+(1+q^{2j})q^{(\ell-2)(2j)+2j-1}\left(\rnh{i}{2}{j-1}{J} + \rnh{i}{4}{j-1}{J} + \dots + \rnh{i}{k-\ell}{j-1}{J} \right).
\end{align}
We summarize the recursions (\ref{rec1})-(\ref{rec2}) as: 
    \begin{align}\label{hrecursion}
       \rnh{i}{\ell}{j}{J}&=q^{2j \left( \ell - 1\right)} \left( q^{2j - 1} \sum \limits^{k - \ell }_{\substack{{m=1} \\ {m \equiv \ell + k\bmod{2}}}} \rnh{i}{m}{j-1}{J}  +  \sum \limits^{k - \left( \ell - 1 \right)}_{\substack{{m=1} \\ {m \not \equiv \ell + k \bmod{2}}}} \rnh{i}{m}{j-1}{J} \right) \nonumber\\& + (1-\delta_{\ell,1})q^{2j \left( \ell - 2\right) } \left( q^{2j - 1} \sum \limits^{k - \ell}_{\substack{{m=1} \\ {m  \equiv \ell + k \bmod{2}}}} \rnh{i}{m}{j-1}{J} + \sum \limits^{k - \left( \ell - 1 \right)}_{\substack{{m=1} \\ {m \not \equiv \ell + k \bmod{2}}}} \rnh{i}{m}{j-1}{J} \right),
    \end{align}
where $\delta_{i,j}$ is the Kronecker delta and $1 \le \ell \le k$.
Moreover, we note that $\rnh{i}{\ell}{J}{J}=\delta_{i,\ell}$ and that
\begin{equation}\label{h1up}
{}^J_ih^{(J+1)}_{\ell} = 
    \begin{cases}
        q^{(\ell - 1)(2J+2)} + q^{(\ell -2)(2J+2)}(1-\delta_{\ell,1}) & \text{if} \quad \ell \leq k - i + 1  \quad \text{and} \quad \ell + k - 1 \equiv i \bmod 2\\ 
        q^{2J+1}\left(  q^{(\ell - 1)(2J+2)} + q^{(\ell -2)(2J+2)}(1-\delta_{\ell,1}) \right) & \text{if} \quad \ell \leq k - i \quad \text{and} \quad \ell + k - 1 \not \equiv i \bmod 2\\
        0 & \text{if} \quad \ell > k - i + 1.
    \end{cases}
\end{equation}
We now have the following proposition, which follows immediately.
\begin{proposition}
    The polynomials $\rnh{i}{\ell}{j}{J}$are uniquely determined by the initial conditions
    \begin{equation}
        \rnh{i}{\ell}{J}{J}= \delta_{i,\ell}
    \end{equation}
    and the recursions
    \begin{align}\label{hMatrixRecursions}
       \rnh{i}{\ell}{j}{J}&=q^{2j \left( \ell - 1\right)} \left( q^{2j - 1} \sum \limits^{k - \ell }_{\substack{{m=1} \\ {m \equiv \ell + k\bmod{2}}}} \rnh{i}{m}{j-1}{J} +  \sum \limits^{k - \left( \ell - 1 \right)}_{\substack{{m=1} \\ {m \not \equiv \ell + k \bmod{2}}}} \rnh{i}{m}{j-1}{J} \right) \nonumber\\& + (1-\delta_{\ell,1})q^{2j \left( \ell - 2\right) } \left( q^{2j - 1} \sum \limits^{k - \ell}_{\substack{{m=1} \\ {m  \equiv \ell + k \bmod{2}}}} \rnh{i}{m}{j-1}{J} + \sum \limits^{k - \left( \ell - 1 \right)}_{\substack{{m=1} \\ {m \not \equiv \ell + k \bmod{2}}}} \rnh{i}{m}{j-1}{J} \right)
    \end{align}
    for $j \geq J+1$.
\end{proposition}

\section{Combinatorial interpretation of the $G_{\ell}$ and $\tilde{G_{\ell}}$}

In this section, we finish our motivated proof and give a combinatorial interpretation for the series $G_{(k-1)J+i}$ with $1 \le i \le k$. We also give combinatorial interpretations for the ghost series $\tilde{G}_{(k-1)J+i}$ with $2 \le \ell \le k$. In order to obtain combinatorial interpretations for our series $G_{(k-1)J+i}$ and $\tilde{G}_{(k-1)J+i}$, we first give a combinatorial interpretation of the polynomials  
$\rnh{i}{\ell}{j}{J}$.

\begin{proposition}\label{hcombinatorics}
    For $j \ge J+1$ and $1\leq i,\ell \leq k$, the polynomial $\rnh{i}{\ell}{j}{J}$ is the generating function for partitions $\lambda=(b_1,\dots,b_s)$, with $b_1\geq\hdots\geq b_s$, satisfying the conditions:
    \begin{enumerate}
        \item No odd parts are repeated,
        \item $f_{2J+1}(\lambda) + f_{2J+2}(\lambda)  \leq k-i$,
        \item  $f_{2t}(\lambda)  + f_{2t+1}(\lambda)  + f_{2t+2}(\lambda)  \leq k-1$ for all $t \ge 0$,
        \item if $f_{2t}(\lambda)  + f_{2t+1}(\lambda)  + f_{2t+2}(\lambda)  = k-1$, then
        \begin{equation*}tf_{2t}(\lambda)  + \left( t+1\right) \left(f_{2t+1}(\lambda)  + f_{2t+2}(\lambda) \right)  \equiv \left( k-1\right)J + k - i + V^o_{\lambda}(t) \bmod{2},
        \end{equation*}
        \item the smallest part $b_s >2J$,
        \item $V^o_{\lambda}(j)\equiv$
        $\begin{cases}
            0\bmod 2\quad\text{if}\quad\ell+(k-1)(j-J)\equiv i\bmod 2 \\
            1\bmod 2\quad\text{if}\quad\ell+(k-1)(j-J)\not\equiv i\bmod 2,
        \end{cases}$ \\
    \item the largest part $b_1 \leq 2j$,
    \item $f_{2j}(\lambda) \in \{ \ell - 1, \ell - 2 \} \cap \mathbb{N} $.
    \end{enumerate}
\end{proposition}

\begin{proof}
 Let $\rnht{i}{\ell}{j}{J}$ denote the generating function for partitions $\lambda$ satisfying conditions 1-8 of the proposition. We verify that $\rnht{i}{\ell}{J+1}{J} = \rnh{i}{\ell}{J+1}{J}$ and that, for $j\geq J + 2$,  $\rnht{i}{\ell}{j}{J}$ satisfies the recursion (\ref{hMatrixRecursions}).  
 In the case $j=J+1$, we have $f_{2J+2}(\lambda)\in\{\ell-2,\ell-1\} \cap \mathbb{N}$, and $f_{2J+1}(\lambda)=0$ or $f_{2J+1}(\lambda)=1$. First, consider when $f_{2J+1}(\lambda)=0$. If $f_{2J+2}(\lambda)=\ell-1$, then, by condition 2, we have $\ell\leq k-i+1$. Similarly, if $f_{2J+2}(\lambda)=\ell-2$, then $\ell\leq k-i+2$. However, if $\ell=k-i+2$, then $$(k-i+2)+k-1\equiv -i+1\not\equiv i\bmod 2,$$ which implies $f_{2J+1}(\lambda)=1$ by condition 6, a contradiction. So, we must have $\ell\leq k-i+1$. In the case when $\ell=k$ and $i=1$, condition 4 is satisfied because $$(J+1)(k-1)=(k-1)J+k-1+V^o_{\lambda}(J)=(k-1)J+k-1+V^o_{\lambda}(J+1)$$ since $V^o_{\lambda}(J) = V^o_{\lambda}(J+1) = 0$. Thus, if $\lambda$ satisfies conditions 1-8 and $f_{2J+1}(\lambda)=0$, then $\lambda$ is counted by \begin{equation}\label{a}
q^{(\ell-1)(2J+2)}+q^{(\ell-2)(2J+2)}.
\end{equation} 
We note that in the case $\ell=1$ there is no case where $f_{2J+2}(\lambda) = \ell-2$ and so in this case $\lambda$ is counted by the first term in (\ref{a}).
Conversely, we can see immediately that if $\lambda$ is counted by (\ref{a}), then it is also counted by $\rnht{i}{\ell}{J+1}{J}$. 

Now, consider when $f_{2J+1}(\lambda)=1$. If $f_{2J+2}(\lambda)=\ell-1$, then $\ell\leq k-i$, and if $f_{2J+2}(\lambda)=\ell-2$, then $\ell\leq k-i+1$. However, if $f_{2J+2}(\lambda)=k-i+1$, then we arrive at a contradiction since $$(k-i+1)+k-1\equiv i\bmod 2$$ implies $f_{2J+1}(\lambda)=0$ by condition 6. Thus, $\ell\leq k-i$. As before, if $\ell=k$ and $i=1$, condition 4 is satisfied. Therefore, if $\lambda$ satisfies conditions 1-8 and if $f_{2J+1}(\lambda)=1$, then $\lambda$ is counted by 
\begin{equation}\label{b}
q^{2J+1}(q^{(\ell-1)(2J+2)}+q^{(\ell-2)(2J+2)}).
\end{equation}
We note that in the case $\ell=1$ there is no case where $f_{2J+2}(\lambda) = \ell-2$ and so in this case $\lambda$ is counted by the first term in (\ref{b}).
Conversely, if $\lambda$ is counted by (\ref{b}), then it is also counted by $\rnht{i}{\ell}{J+1}{J}$. We conclude that
\begin{equation}
\rnht{i}{\ell}{J+1}{J} = 
    \begin{cases}
        q^{(\ell - 1)(2J+2)} + q^{(\ell -2)(2J+2)}(1-\delta_{\ell,1}) & \text{if} \quad \ell \leq k - i + 1  \quad \text{and} \quad \ell + k - 1 \equiv i \bmod 2\\ 
        q^{2J+1}\left(  q^{(\ell - 1)(2J+2)} + q^{(\ell -2)(2J+2)}(1-\delta_{\ell,1})\right) & \text{if} \quad \ell \leq k - i \quad \text{and} \quad \ell + k - 1 \not \equiv i \bmod 2\\
        0 & \text{if} \quad \ell > k - i + 1,
    \end{cases}
\end{equation}
which agrees with $\rnh{i}{\ell}{J+1}{J}$ by inspection of (\ref{h1up}). 

Now, consider $j \geq J + 2$. The partitions $\lambda$ satisfying conditions 1-8 of the proposition can be divided into two sets: those where $f_{2j-1}(\lambda) = 0$ and those where $f_{2j-1}(\lambda) = 1$.

First, we consider when $f_{2j-1}(\lambda) = 0$. In this case, the partitions in question have either the form
$$\left( (2j)^{\ell - 1}, b_{\ell}, \hdots, b_s \right)  \quad \text{or} \quad \left( (2j)^{\ell - 2}, b_{\ell -1}, \hdots, b_s \right), $$ where  $\lambda_1 = \left( b_{\ell}, \hdots, b_s \right)$ and $\lambda_2 = \left( b_{\ell - 1}, \hdots, b_s \right)$ are partitions satisfying the first 6 conditions of the proposition having largest part at most $2j-2$. By condition $3$ of the proposition, when $f_{2j}(\lambda)=\ell-2$, it is clear that $f_{2j-2}(\lambda)\leq k- (\ell -1)$. Notice that if $f_{2j-2}(\lambda) = k - (\ell - 1 )$, then, by condition 4, we have
    \begin{align*}
        (j-1)f_{2j-2}(\lambda)+j(f_{2j-1}(\lambda)+f_{2j}(\lambda))&= (j-1)(k-\ell+1) + j(\ell-2)\\
        &=\ell + (k-1)j-k - 1 \\
        &\equiv (k-1)J+k-i+V^o_{\lambda}(j-1)\bmod 2,
    \end{align*}
    and thus,
    \begin{equation*}
        \hspace{3.8em} \ell + (k-1)(j-J) \equiv 1-i+V^o_{\lambda}(j-1)\bmod 2.
    \end{equation*}
However, condition 6 now implies $ i \equiv 1 -i  \bmod 2$, a contradiction. Thus, we have $f_{2j-2}(\lambda)\leq k- \ell$. When $f_{2j}(\lambda)=\ell-1$, we notice that if $f_{2j-2}(\lambda)= k - \ell$, then by condition 4, we have
    \begin{align*}
       (j-1)f_{2j-2}(\lambda)+j(f_{2j-1}(\lambda)+f_{2j}(\lambda))&=(j-1)(k-\ell) + j(\ell-1)\\
        &=\ell+(k-1)j-k\\
        &\equiv (k-1)J+k-i+V^o_{\lambda}(j-1)\bmod 2, 
    \end{align*}
    and thus,
    \begin{equation*}
       \hspace{2.8em} \ell+(k-1)(j-J) \equiv-i+V^o_{\lambda}(j-1)\bmod 2,
    \end{equation*}
which is valid by condition 6 since $V^o_\lambda(j-1) = V^o_\lambda(j)$. Thus, condition 4 is satisfied and we have that $f_{2j-2}(\lambda)\leq k- \ell$. Therefore, for $r=1,2$, we can see that $\lambda_r$ is counted by some $\rnht{i}{m}{j-1}{J}$ for $1\leq m\leq k-(\ell-1)$. Since it is always the case that $V^o_\lambda(j) = V^o_{\lambda_r}(j-1)$, we have
$V^o_\lambda(j) \equiv V^o_{\lambda_r}(j-1) \bmod 2$, which is, by condition 6, equivalent to 
   \begin{equation}
       \ell +(k-1)(j-J) \equiv m + (k-1)(j-1-J) \bmod 2.
   \end{equation}
However, this can be rewritten as $m\not\equiv\ell+k\bmod 2$, and so we have that $\lambda_1$ and $\lambda_2$ are counted by $\rnht{i}{m}{j-1}{J}$, where $m \not \equiv \ell + k \bmod 2$. Hence, if $\lambda$ satisfies conditions 1-8 with $f_{2j-1}(\lambda) = 0 $, then $\lambda$ is counted by
   \begin{equation}\label{recHalf1}
        q^{2j \left( \ell - 1\right)} \sum \limits^{k - \left( \ell - 1 \right)}_{\substack{{m=1} \\ {m \not \equiv \ell + k \bmod{2}}}} \rnht{i}{m}{j-1}{J} + q^{2j \left( \ell - 2\right)} \sum \limits^{k - \left( \ell - 1 \right)}_{\substack{{m=1} \\ {m \not \equiv \ell + k \bmod{2}}}} \rnht{i}{m}{j-1}{J}.
   \end{equation}
Note that if $\ell=1$, then there is no case where $f_{2j-2}(\lambda)=\ell-2$. Thus, if $\ell=1$, $\lambda$ is counted by the first sum in (\ref{recHalf1}). 
Conversely, assume $\lambda$ is counted by $\rnht{i}{m}{j-1}{J}$ for some $1\leq m\leq k-(\ell-1)$ that satisfies $m \not \equiv\ell + k \bmod 2$. If $f_{2j-2}(\lambda)=m-2$, then it follows immediately that $((2j)^{\ell-1},\lambda)$ and $((2j)^{\ell-2},\lambda)$ are counted by $\rnht{i}{\ell}{j}{J}$. Similarly, if $f_{2j-2}(\lambda)=m-1$, then $((2j)^{\ell-2},\lambda)$ is counted by $\rnht{i}{\ell}{j}{J}$. We can also see that $((2j)^{\ell-1},\lambda)$ is counted by $\rnht{i}{\ell}{j}{J}$ because when $m=k-(\ell-1)$, condition 4 is satisfied.

Next, consider partitions where $f_{2j-1}(\lambda) = 1$. In this case, the partitions in question have either the form $$\left( (2j)^{\ell - 1}, 2j-1, b_{\ell + 1}, \hdots, b_s \right)\quad \text{or} \quad\left( (2j)^{\ell - 2}, 2j-1, b_{\ell}, \hdots, b_s \right), $$ where $\lambda_1'=\left( b_{\ell + 1}, \hdots, b_s \right)$ and $\lambda_2'=\left( b_{\ell}, \hdots, b_s \right)$ are partitions satisfying the first 6 conditions of the proposition with largest part at most $2j - 2$. When $f_{2j}(\lambda) = \ell - 2$, we  have, from condition 3, that $f_{2j-2}(\lambda) \leq k - \ell$; however, as before, if $f_{2j-2}(\lambda) = k -\ell $, then condition $4$ is violated, and thus, we have $f_{2j-2}(\lambda) \leq k- \ell - 1$. 
Furthermore, by condition 3, we have that if $f_{2j}(\lambda) = \ell - 1 $, then $ f_{2j-2}(\lambda) \leq k - \ell - 1$, and, as above, we have that condition $4$ is satisfied when $ f_{2j-2}(\lambda) = k - \ell - 1$. Thus, $f_{2j-2}(\lambda)\leq k - \ell - 1$. Therefore, for $r=1,2$, $\lambda_r'$ is counted by some $\rnht{i}{m}{j-1}{J}$ for $1\leq m\leq k-\ell$. Since it is always the case that $V^o_\lambda(j) = V^o_{\lambda_r'}(j-1) + 1$, we have that $V^o_\lambda(j) \equiv V^o_{\lambda_r'}(j-1) + 1\bmod 2$, which is, by condition 6, equivalent to 
    \begin{equation}
        \ell +(k-1)(j-J) \equiv m + (k-1)(j-1-J) + 1 \bmod 2
    \end{equation}
for $1\leq m\leq k-\ell$. However, this can be rewritten as $m \equiv\ell+k\bmod 2$, and so, $\lambda_1'$ and $\lambda_2'$ are counted by $\rnht{i}{m}{j-1}{J}$, where $m \equiv \ell + k \bmod 2$. Hence, if $\lambda$ satisfies conditions 1-8 with $f_{2j-1}(\lambda) = 1 $, then $\lambda$ is counted by
     \begin{equation}\label{recHalf2}
        q^{2j \left( \ell - 1\right) + 2j - 1} \sum \limits^{k - \ell }_{\substack{{m=1} \\ {m  \equiv \ell + k \bmod{2}}}} \rnht{i}{m}{j-1}{J} + q^{2j \left( \ell - 2\right) + 2j - 1} \sum \limits^{k - \ell}_{\substack{{m=1} \\ {m \equiv \ell + k\bmod{2}}}} \rnht{i}{m}{j-1}{J}.
     \end{equation}
Again, note that if $\ell=1$, there is no case where $f_{2j-2}(\lambda)=\ell-2$. Thus, if $\ell=1$, $\lambda$ is counted by the first sum in (\ref{recHalf2}). Conversely, assume $\lambda$ is counted by $\rnht{i}{m}{j-1}{J}$ for some $1\leq m\leq k-\ell$ that satisfies $m \equiv\ell + k\bmod 2$. If $f_{2j-2}(\lambda)=m-2$, then it follows immediately that $((2j)^{\ell-1},2j - 1, \lambda)$ and $((2j)^{\ell-2},2j - 1, \lambda)$ are counted by $\rnht{i}{\ell}{j}{J}$. Similarly, if $f_{2j-2}(\lambda)=m-1$, then $((2j)^{\ell-2},2j-1, \lambda)$ is counted by $\rnht{i}{\ell}{j}{J}$. We can also see that $((2j)^{\ell-1},2j-1 ,\lambda)$ is counted by $\rnht{i}{\ell}{j}{J}$ because when $m=k-\ell$, condition 4 is satisfied.

Thus, we conclude that the partitions satisfying conditions 1-8 are counted by 
    \begin{align}
        \rnht{i}{\ell}{j}{J}&=q^{2j \left( \ell - 1\right)} \left( q^{2j - 1} \sum \limits^{k - \ell }_{\substack{{m=1} \\ {m \equiv \ell + k\bmod{2}}}} \rnht{i}{m}{j-1}{J}  +  \sum \limits^{k - \left( \ell - 1 \right)}_{\substack{{m=1} \\ {m \not \equiv \ell + k \bmod{2}}}} \rnht{i}{m}{j-1}{J} \right) \nonumber\\& + (1-\delta_{\ell,1})q^{2j \left( \ell - 2\right) } \left( q^{2j - 1} \sum \limits^{k - \ell}_{\substack{{m=1} \\ {m  \equiv \ell + k \bmod{2}}}} \rnht{i}{m}{j-1}{J} + \sum \limits^{k - \left( \ell - 1 \right)}_{\substack{{m=1} \\ {m \not \equiv \ell + k \bmod{2}}}} \rnht{i}{m}{j-1}{J} \right),
    \end{align}
which is the recursion (\ref{hMatrixRecursions}). This proves the proposition. 

\end{proof}

Next, as an important step in our motivated proof of the Bressoud-G{\"o}llnitz-Gordon identities, we will need the following result: \begin{corol}\label{h1h2}
For $j\geq J+1$ and $1\leq i\leq k$, the polynomial $\rnh{i}{1}{j}{J}+\rnh{i}{2}{j}{J}$ is the generating function for partitions $\lambda=(b_1,\dots,b_s)$, with $b_1\geq\hdots\geq b_s$, satisfying the conditions: 
    \begin{enumerate}
        \item No odd parts are repeated,
        \item $f_{2J+1}(\lambda) + f_{2J+2}(\lambda)  \leq k-i$,
        \item $f_{2t}(\lambda)  + f_{2t+1}(\lambda)  + f_{2t+2}(\lambda)  \leq k-1$ for all $t \ge 0$,
        \item if $f_{2t}(\lambda)  + f_{2t+1}(\lambda)  + f_{2t+2}(\lambda)  = k-1$, then 
        \begin{equation*}
            tf_{2t}(\lambda)  + \left( t+1\right) \left(f_{2t+1}(\lambda)  + f_{2t+2}(\lambda) \right)  \equiv \left( k-1\right)J + k - i + V^o_{\lambda}(t) \bmod{2}, 
        \end{equation*}
        \item the smallest part $b_s >2J$,
    \item the largest part $b_1 \leq 2j$,
    \item $f_{2j}(\lambda) \in \{0,1\}$.
                \end{enumerate}
\end{corol}
\begin{proof}
Consider the partitions counted by $\rnh{i}{1}{j}{J}$. We can see that condition 6 in Proposition \ref{hcombinatorics} becomes 
    \begin{equation}
        V^o_{\lambda}(j)\equiv
        \begin{cases}
            0\bmod 2\quad\text{if}\quad 1+(k-1)(j-J)\equiv i\bmod 2 \\
            1\bmod 2\quad\text{if}\quad 1+(k-1)(j-J)\not\equiv i\bmod 2.
        \end{cases} \vspace{2mm}\\ 
    \end{equation} 
If $1+(k-1)(j-J)\equiv i\bmod 2$, then $\rnh{i}{1}{j}{J}$ counts partitions of even integers. This implies $2+(k-1)(j-J)\not\equiv i\bmod 2$, and so $\rnh{i}{2}{j}{J}$ counts partitions of odd integers. Similarly, if $1+(k-1)(j-J)\not\equiv i\bmod 2$, then $\rnh{i}{1}{j}{J}$ counts partitions of odd integers and $\rnh{i}{2}{j}{J}$ counts partitions of even integers. Thus, we can see immediately from Proposition \ref{hcombinatorics} that $\rnh{i}{1}{j}{J}+\rnh{i}{2}{j}{J}$ counts the partitions $\lambda$ satisfying conditions 1-7 of the Corollary. 
    \end{proof}

Now, we complete our motivated proof of the Bressoud-G{\"o}llnitz-Gordon identities by giving a combinatorial interpretation for the series $G_{(k-1)J+i}$. The identities follow when we set $J=0$. In this proof, we use the definition of the limit of a sequence of series in $\mathbb{C}[[q]]$, which we now recall from \cite{CKLMQRS}.  Let $\{A_j(q)\}_{j=0}^\infty$ be a sequence of elements of $\mathbb{C}[[q]]$. We say that 
\begin{equation}
    \lim_{j \to \infty} A_j(q)
\end{equation}
exists if, for each $m \ge 0$, there is some $J_m > 0$ such that the coefficients of $q^m$ in each series $A_j(q)$ are equal for $j \ge J_m$. In other words, the limit exists if the coefficients of each power of $q$ stabilize as $j \to \infty$. We now have the following result:
\begin{theo}\label{GCombinatorics}
For $1\leq i\leq k$, the power series $G_{(k-1)J+i}$ is the generating function for partitions  $\lambda=(b_1,\dots,b_s)$, with $b_1\geq\hdots\geq b_s$, satisfying the conditions: 
    \begin{enumerate}
        \item No odd parts are repeated,
        \item $f_{2J+1}(\lambda) + f_{2J+2}(\lambda)  \leq k-i$,
        \item $f_{2t}(\lambda)  + f_{2t+1}(\lambda)  + f_{2t+2}(\lambda)  \leq k-1$ for all $t \ge 0$,
        \item if $f_{2t}(\lambda)  + f_{2t+1}(\lambda)  + f_{2t+2}(\lambda)  = k-1$, then
        \begin{equation*}
            tf_{2t}(\lambda)  + \left( t+1\right) \left(f_{2t+1}(\lambda)  + f_{2t+2}(\lambda) \right)  \equiv \left( k-1\right)J + k - i + V^o_{\lambda}(t) \bmod{2},
        \end{equation*}
        \item the smallest part $b_s >2J$.
    \end{enumerate}
\end{theo}
\begin{proof}
For each $1\leq i\leq k$ and $j\geq J+1$, we have
    \begin{equation}
        G_{(k-1)J+i}=\sum_{n=1}^{k}\rnh{i}{n}{j}{J}G_{(k-1)j+n}.
    \end{equation}
First, we show that 
\begin{equation}\label{lim1}
    \lim_{j\to\infty}\rnh{i}{\ell}{j}{J}=0
\end{equation}
for all $\ell>2$. For $m\geq 0$, consider the coefficient of $q^m$ in $\rnh{i}{\ell}{j}{J}$. Taking $J_m>\frac{m}{2}$, we can see that if $j\geq J_m$, then the coefficient of $q^m$ in $\rnh{i}{\ell}{j}{J}$ is $0$,
thus proving (\ref{lim1}).
Since the coefficients of $G_{(k-1)J+i}$ are stable as $j\to\infty$, the following limit exists:
    \begin{align}
        G_{(k-1)J+i}=\lim_{j\to\infty}\sum_{n=1}^{k}\rnh{i}{n}{j}{J}G_{(k-1)j+n}.
    \end{align}
Furthermore, by (\ref{lim1}) and the Empirical Hypothesis, we can see that 
    \begin{equation*}
        \lim_{j\to\infty}\sum_{n=3}^{k}\rnh{i}{n}{j}{J}G_{(k-1)j+n}=\sum_{n=3}^{k}\lim_{j\to\infty}\rnh{i}{n}{j}{J}\lim_{j\to\infty}G_{(k-1)j+n}=\sum_{n=3}^{k}0 \cdot 1=0
    \end{equation*}
because
    \begin{equation}
        \lim_{j \to \infty} G_{(k-1)j+i} = 1
    \end{equation}
for $1 \le i \le k$.
So, we have 
    \begin{align*}
        G_{(k-1)J+i}&=\lim_{j\to\infty}\sum_{n=1}^{k}\rnh{i}{n}{j}{J}G_{(k-1)j+n}-\lim_{j\to\infty}\sum_{n=3}^{k}\rnh{i}{n}{j}{J}G_{(k-1)j+n}\\
        &=\lim_{j\to\infty}\left(\rnh{i}{1}{j}{J}G_{(k-1)j+1}+\rnh{i}{2}{j}{J}G_{(k-1)j+2}\right)\\
        &=\lim_{j\to\infty}\left(\rnh{i}{1}{j}{J}+\rnh{i}{2}{j}{J}+q^{2j+1}\gamma(q)\right)
    \end{align*}
for some $\gamma(q)\in\mathbb{C}[[q]]$ by the Empirical Hypothesis. We can see that 
    \begin{equation}
        \lim_{j\to\infty}q^{2j+1}\gamma(q)=0
    \end{equation}
because, for any $m\geq 0$, taking $J_m>\frac{m-1}{2}$, $j\geq J_m$ implies that the coefficient of $q^m$ in $q^{2j+1}\gamma(q)$ is $0$. 
Thus, we have
    \begin{align}
        G_{(k-1)J+i}=\lim_{j\to\infty}\left(\rnh{i}{1}{j}{J}+\rnh{i}{2}{j}{J}+q^{2j+1}\gamma(q)\right)-\lim_{j\to\infty}q^{2j+1}\gamma(q)=\lim_{j\to\infty}\left(\rnh{i}{1}{j}{J}+\rnh{i}{2}{j}{J}\right).
    \end{align}
We conclude that as $j\to\infty$, conditions 6 and 7 of Corollary \ref{h1h2} vanish, thus proving the theorem.
\end{proof}

Finally, we give a combinatorial interpretation of the ghost series, $\tilde{G}_{\ell}$.
\begin{theo}\label{GhostCombinatorics}
For $2\leq i\leq k$, the power series $\Tilde{G}_{(k-1)J+i}$ is the generating function for partitions $\lambda=(b_1,\dots,b_s)$, with $b_1\geq\hdots\geq b_s$, satisfying the conditions: 
    \begin{enumerate}
        \item No odd parts are repeated,
        \item $f_{2J+1}(\lambda)  + f_{2J+2}(\lambda)   \leq k-i$,
        \item $f_{2t}(\lambda)   + f_{2t+1}(\lambda)   + f_{2t+2}(\lambda)   \leq k-1$ for all $t \ge 0$,
        \item if $f_{2t}(\lambda)   + f_{2t+1}(\lambda)  + f_{2t+2}(\lambda)   = k-1$, then \begin{equation}\label{paritycon}
        tf_{2t}(\lambda)   + \left( t+1\right) \left(f_{2t+1}(\lambda)   + f_{2t+2}(\lambda) \right)  \equiv \left( k-1\right)J + k - i + 1 + V^o_{\lambda }(t) \bmod{2},
        \end{equation}
        \item the smallest part $b_s >2J$.
    \end{enumerate}
\end{theo}

\begin{proof}
Using (\ref{Gdef2}) - (\ref{Gdef4}), we have for $J \geq 0$ and $2\leq i\leq k$ that
    \begin{align} 
        \Tilde{G}_{(k-1)J+k} &= G_{(k-1)(J+1)+2}\label{GhostComb1} \\
        \Tilde{G}_{(k-1)J+i} &=q^{(2J+2)(k-i)}G_{(k-1)(J+1)+k-i+2}+q^{(2J+2)(k-i-1)+2J+1}G_{(k-1)(J+1)+k-i+1} \nonumber \\
        &\hphantom{=}+G_{(k-1)J+i+1} \label{GhostComb2}.
    \end{align}
We first consider the case of  $\Tilde{G}_{(k-1)J+k}$. Let $\lambda$ be a partition counted by 
    \begin{equation*}
        G_{(k-1)(J+1)+2}.
    \end{equation*}
It is immediate from Theorem \ref{GCombinatorics} that $\lambda$ satisfies conditions $1,2,3$, and $5$ of our theorem. Moreover, we note that the parity condition (\ref{paritycon}) of our theorem is the same as that which $\lambda$ has from Theorem \ref{GCombinatorics} since, for all $t \ge 0$, $$(k+1)J+k-k+1+V^o_{\lambda}(t) \equiv (k-1)(J+1)+k-2+V^o_{\lambda}(t) \bmod{2}.$$ Hence, $\lambda$ satisfies the conditions of the theorem.

Now, we consider the case of ${\Tilde{G}}_{(k-1)J+i}$. Let $\lambda$ be a partition counted by the right-hand side of (\ref{GhostComb2}).  It is immediate from Theorem \ref{GCombinatorics} that $\lambda$ satisfies conditions $1,2$, and $5$ of our theorem. We verify that $\lambda$ satisfies conditions $3$ and $4$ of our theorem for each of the cases. 
    
First, we consider partitions $\lambda$ counted by 
    \begin{equation}\label{GhostPart1}
        q^{(2J+2)(k-i)}G_{(k-1)(J+1)+k-i+2}.
    \end{equation}
By condition 2 of Theorem \ref{GCombinatorics}, we have $f_{2J+3}( \lambda)+f_{2J+4}( \lambda)\le i-2$. Since $f_{2J+2}( \lambda)=k-i$, we have $f_{2J+2}(\lambda)+f_{2J+3}(\lambda)+f_{2J+4}(\lambda)\leq k-2$. Thus, condition 3 of our theorem is satisfied. Additionally, we note that $G_{(k-1)(J+1)+k-i+2}$ has the same parity condition as that of our theorem since, for all $t\geq 0$, we have
  \begin{align*}
      (k-1)J+k-i+1 + V^o_{\lambda}(t) &\equiv (k-1)(J+1)+k-(k-i+2) + V^o_{\lambda_1}(t) \bmod{2},
  \end{align*}
where $\lambda_1$ is a partition counted by $G_{(k-1)(J+1)+k-i+2}$. Further, $V^o_{\lambda_1}(t)$ remains unchanged by the addition of the $k-i$ parts $2J+2$. Hence, $\lambda$ satisfies the conditions of our theorem. 

Next, consider partitions $\lambda$ counted by 
    \begin{equation}\label{GhostPart2}
        q^{(2J+2)(k-i-1)+2J+1}G_{(k-1)(J+1)+k-i+1}.
    \end{equation}
By condition 2 of Theorem \ref{GCombinatorics}, we have $f_{2J+3}(\lambda)+f_{2J+4}(\lambda)\leq i-1$, and since $f_{2J+2}(\lambda)=k-i-1$, it follows that $f_{2J+2}(\lambda)+f_{2J+3}(\lambda)+f_{2J+4}(\lambda)\leq k-2$ so that condition 3 in our theorem is satisfied. Here, we note that partitions counted by $G_{(k-1)(J+1)+k-i+1}$ do not initially have the same the parity condition as our theorem since, for all $t \ge 0$,
    \begin{align*}
        (k-1)J+k-i+1 + V^o_{\lambda}(t)&\not \equiv (k-1)(J+1)+k-(k-i+1)+ V^o_{\lambda_2}(t) \bmod{2},
    \end{align*}
where $\lambda_2$ is a partition counted by $G_{(k-1)(J+1)+k-i+1}$. However, this is corrected with the addition of the part $2J+1$. So, $\lambda$ satisfies the conditions of our theorem. 

Finally, assume $\lambda$ is counted by $G_{(k-1)J+i+1}$. Condition $3$ of our theorem is immediate. Further, we have that $G_{(k-1)J+i+1}$ has the same parity condition as that of our theorem since, for all $t \ge 0$,
   \begin{align*}
        (k-1)J+k-i+1 + V^o_{\lambda}(t) &\equiv (k-1)J+k-(i+1)+ V^o_{\lambda}(t) \bmod{2}.
    \end{align*}
Hence, $\lambda$ satisfies the conditions of our theorem.

The converse follows similarly by removing the parts $2J+1$ and $2J+2$ from a partition counted by $\tilde{G}_{(k-1)J+i}$.
\end{proof}

\begin{rema}(cf. Remark 2.1 in \cite{LZ}, Remark 7.6 in \cite{KLRS}, Remark 6.6 in \cite{CKLMQRS})
    We note that as discussed in \cite{AB}, \cite{R}, and \cite{A4}, an alternate proof of Theorem \ref{GCombinatorics} and Theorem \ref{GhostCombinatorics} which uses only the Empirical Hypothesis and does not use the combinatorial interpretation of the polynomials 
    $\rnh{i}{n}{j}{J}$ may be given as follows: Let $H_1,H_2,\dots$ and $\tilde{H}_1, \tilde{H}_2,\dots$ be sequences of formal power series satisfying (\ref{Gdef1}) - (\ref{Gdef4}) and the Empirical Hypothesis (with $H_\ell$ in place of $G_\ell$ and $\tilde{H}_\ell$ in place of $\tilde{G}_\ell$ for $\ell \ge 1$). It follows that the $H_\ell$ and $\tilde{H}_\ell$ are uniquely determined by these recursions and the Empirical Hypothesis. It is easy to see that the generating functions counting the conditions in Theorem \ref{GCombinatorics} and Theorem \ref{GhostCombinatorics} satisfy the recursions and Empirical Hypothesis as well. We also have proved that the $G_{\ell}$ and $\tilde{G}_{\ell}$ satisfy these recursions and Empirical Hypothesis as well. Therefore, by uniqueness, we have Theorem \ref{GCombinatorics} and Theorem \ref{GhostCombinatorics}. 
\end{rema}

\begin{rema}\label{Ghost1}
    Although it is not necessary in our proof, we note that, as in Remark 7.5 in \cite{KLRS}, we can define the ghost by 
    \begin{equation}\label{ghost1}
        \tilde{G}_1 = G_2. 
    \end{equation}
    Indeed, extending Theorem \ref{GhostCombinatorics},  we see that conditions 1-5 in Theorem \ref{GhostCombinatorics} with $J=0$ and $i=1$ agree precisely with conditions 1-5 in Theorem \ref{GCombinatorics} with $J=0$ and $i=2$. Additionally, (\ref{ghost1}) can also be understood using the left-hand and right-hand sides of (\ref{Gdef2}) with $j=-1$. We explore this phenomenon in more generality in the next section.
\end{rema}

\section{An $(a;x;q)$-dictionary for the Bressound-G{\"o}llnitz-Gordon identities and the ghost series}

In this section, we establish a dictionary between our series on various shelves with the series $\tilde{J}_{k,i}(a;x;q)$ defined in \cite{CoLoMa} with variables specialized appropriately. We also give ``ghost series" corresponding to $\tilde{J}_{k,i}(a;x;q)$ and explore their properties. Recall from \cite{CoLoMa} the series 
\begin{equation}\label{H}
    \tilde{H}_{k,i}(a;x;q) =  \sum_{n \ge 0} \frac{(-a)^n q^{kn^2-\binom{n}{2} + n -in}x^{(k-1)n}(1-x^iq^{2ni})(-x;q)_n(-1/a;q)_n(-axq^{n+1})_\infty}{(q^2;q^2)_n (xq^n;q)_\infty}
\end{equation}
and
\begin{equation}\label{J}
    \tilde{J}_{k,i}(a;x;q) = \tilde{H}_{k,i}(a;xq;q) + axq \tilde{H}_{k,i-1}(a;xq;q).
\end{equation}
Using (\ref{H}), we write (\ref{J}) as 
\begin{align}\label{FullJt}
    \tilde{J}_{k,i}(a;x;q)
    &=\sum_{n \ge 0} \frac{(-a)^n q^{kn^2-\binom{n}{2} + kn -in }x^{(k-1)n}(-xq;q)_n(-1/a;q)_n(-axq^{n+2})_\infty}{(q^2;q^2)_n (xq^{n+1};q)_\infty}\nonumber\\
    &\hspace{1em}\cdot \left(1-x^iq^{(2n+1)i}+axq^{n+1}(1-x^{i-1}q^{(2n+1)(i-1)})\right).
\end{align}

Our dictionary between the series $G_\ell$ for $\ell \ge 1$ in the current work and the series $\tilde{J}_{k,i}(a;x;q)$ in \cite{CoLoMa} is thus given by the following result:
\begin{proposition}
For $1 \le i \le k$, we have
\begin{equation}\label{dictionary}
    G_{(k-1)j+i} = \tilde{J}_{k,k-i+1}(1/q;q^{2j};q^2).
\end{equation}
\end{proposition}
\begin{proof}
This follows immediately from (\ref{FullJt}). Indeed, we have
\begin{align}
    &\tilde{J}_{k,k-i+1}(1/q;q^{2j};q^2)\\
    &=\sum_{n \ge 0} \frac{(-1/q)^n q^{2kn^2-2\binom{n}{2} + 2kn -2(k-i+1)n }q^{2j(k-1)n}(-q^{2j+2};q^2)_n(-q;q^2)_n(-q^{2n+2j+3};q^2)_\infty}{(q^4;q^4)_n (q^{2n+2j+2};q^2)_\infty}\\
    &\hspace{1em}\cdot (1-q^{2(2n+j+1)(k-i+1)}+q^{2n+2j+1}(1-q^{2(2n+j+1)(k-i)}).
\end{align} 
Noticing that
\begin{align*}
    &\frac{(-q^{2j+2};q^2)_n (-q;q^2)_n (-q^{2n+2j+3};q^2)_\infty}{(q^4;q^4)_n(q^{2n+2j+2};q^2)_\infty}\\ 
    &=\frac{(-q;q^2)_n(-q^{2n+1};q^2)_{j+1}(-q^{2n+2j+3};q^2)_\infty(-q^{2j+2};q^2)_n(q^{2n+2};q^2)_j}{(q^2;q^2)_n(-q^2;q^2)_n(-q^{2n+1};q^2)_{j+1}(q^{2n+2j+2};q^2)_{\infty}(q^{2n+2};q^2)_j}\\
    &=\frac{(-q;q^2)_\infty (-q^{2j+2};q^2)_n(q^{2n+2};q^2)_j}{(-q^2;q^2)_n(-q^{2n+1};q^2)_{j+1}(q^2;q^2)_\infty}\\
    &=\frac{1}{F(q)}\frac{(-q^{2j+2};q^2)_n(q^{2n+2};q^2)_j}{(-q^2;q^2)_n(-q^{2n+1};q^2)_{j+1}},
\end{align*}
where we recall that $\frac{1}{F(q)} = \frac{(-q;q^2)_\infty}{(q^2;q^2)_\infty}$ now gives us precisely (\ref{dictionary}).
\end{proof}

Before we proceed, we recall some important properties of $\tilde{H}_{k,i}(a;x;q)$ and $\tilde{J}_{k,i}(a;x;q)$. In \cite{CoLoMa}, the following fundamental properties of $\tilde{H}_{k,i}(a;x;q)$ are proved:
\begin{lemma}(\cite{CoLoMa}, Lemma 2.1)
   \begin{align}
        \tilde{H}_{k,0}(a;x;q) \,\ & = \,\ 0 \label{Hproperty1} \\
        \tilde{H}_{k,-i}(a;x;q) \,\ & = \,\ -x^{-i}H_{k,i}(a;x;q) \label{Hproperty2} \\
        \tilde{H}_{k,i}(a;x;q) - \tilde{H}_{k,i-2}(a;x;q) \,\ & = \,\ x^{i-2}(1+x)\tilde{J}_{k,k-i+1}(a;x;q). \label{Hproperty3}
   \end{align}
\end{lemma}
\noindent They also prove the following fundamental properties of $\tilde{J}_{k,i}(a;x;q)$:
\begin{theo}(\cite{CoLoMa}, Theorem 2.2)\label{CoLoMaTheorem}
    \begin{align}
        \tilde{J}_{k,1}(a;x;q) \,\ &= \,\ \tilde{J}_{k,k}(a;xq;q) \label{CoLoMa1}\\
        \tilde{J}_{k,2}(a;x;q) \,\ &= \,\ (1+xq)\tilde{J}_{k,k-1}(a;xq;q) + axq\tilde{J}_{k,k}(a;xq;q) \label{CoLoMa2}\\
        \tilde{J}_{k,i}(a;x;q) - \tilde{J}_{k,i-2}(a;x;q) \,\ &= \,\ (xq)^{i-2}(1+xq)\tilde{J}_{k,k-i+1}(a;xq;q) \label{CoLoMa3}\\
        &+ \,\ a(xq)^{i-2}(1+xq)\tilde{J}_{k,k-i+2}(a;xq;q). \nonumber
    \end{align}
\end{theo}
\noindent We note that, using (\ref{CoLoMa1}), we can rewrite (\ref{CoLoMa2}) as
\begin{equation}\label{CoLoMa2Rewrite}
    \tilde{J}_{k,k-1}(a;xq;q) = \frac{\tilde{J}_{k,2}(a;x;q) - axq\tilde{J}_{k,1}(a;x;q)}{1+xq}
\end{equation}
and that (\ref{CoLoMa3}) can be written as 
    \begin{equation}\label{CoLoMa3Rewrite}
        \tilde{J}_{k,k-i+1}(a;xq;q) = \frac{\tilde{J}_{k,i}(a;x;q) - \tilde{J}_{k,i-2}(a;x;q)}{(xq)^{i-2}(1+xq)} - a\tilde{J}_{k,k-i+2}(a;xq;q).
    \end{equation}
In fact, after specializing $(a;x;q)$ to $(1/q;q^{2j};q^2)$ in (\ref{CoLoMa1}), (\ref{CoLoMa2Rewrite}), and (\ref{CoLoMa3Rewrite}), we obtain (\ref{Grecursion1}) - (\ref{Grecursion3}). The necessity of our ghost series may then be motivated as a way of introducing new series into (\ref{Grecursion1}) - (\ref{Grecursion3}) so that division by only a pure power of $q$ is necessary (that is, we wish to obtain a recursion that does not involve dividing by $1+q^{2j+2}$). In our current setting, we wish to rewrite (\ref{CoLoMa2Rewrite}) and (\ref{CoLoMa3Rewrite}) so that division by $1+xq$ is not necessary.

With the discussion above in mind, we now define an $(a;x;q)$ version of our ghost series. We will use $\tilde{\tilde{J}}_{k,i}(a;x;q)$ to denote the ghost series in this section. Motivated by (\ref{Gdef2}) - (\ref{Gdef4}), we define the series $\tilde{\tilde{J}}_{k,i}(a;x;q)$ for $1 \le i \le k-1$ as follows:
\begin{align}\label{JghostDef1}
    \tilde{J}_{k,k-1}(a;xq;q) = \frac{\tilde{J}_{k,2}(a;x;q) - \tilde{\tilde{J}}_{k,1}(a;x;q)}{xq} - a \tilde{J}_{k,k}(a,xq,q) = \tilde{\tilde{J}}_{k,1}(a;x;q)
\end{align}
and
\begin{align}\label{JghostDef2}
    \tilde{J}_{k,k-i+1}(a;xq;q) &= \frac{\tilde{J}_{k,i}(a;x;q) - \tilde{\tilde{J}}_{k,i-1}(a;x;q)}{(xq)^{i-1}} - a\tilde{J}_{k,k-i+2}(a;xq;q) \nonumber\\
    &=\frac{\tilde{\tilde{J}}_{k,i-1}(a;x;q) - \tilde{J}_{k,i-2}(a;x;q)}{(xq)^{i-2}}- a\tilde{J}_{k,k-i+2}(a;xq;q).
\end{align}
Using (\ref{JghostDef1}) and (\ref{JghostDef2}), the ghosts can be expressed as 
\begin{equation}\label{Jtildetildedef}
    \tilde{\tilde{J}}_{k,i}(a;x;q) = \frac{\tilde{J}_{k,i+1}(a;x;q) + xq\tilde{J}_{k,i-1}(a;x;q)}{1+xq}
\end{equation}
for $2 \le i \le k-1$
and
\begin{equation}\label{Jtildetilde1}
    \tilde{\tilde{J}}_{k,1}(a;x;q) = \frac{\tilde{J}_{k,2}(a;x;q) - axq \tilde{J}_{k,1}(a;x;q)}{1+xq}.
\end{equation}
We note, using (\ref{Hproperty1}) and (\ref{Hproperty2}), that (\ref{Jtildetilde1}) is just
\begin{align*}
    \tilde{\tilde{J}}_{k,1}(a;x;q) &= \frac{\tilde{J}_{k,2}(a;x;q) - axq \tilde{J}_{k,1}(a;x;q)}{1+xq}\\
    &=\frac{\tilde{H}_{k,2}(a;xq;q) + axq\tilde{H}_{k,1}(a;xq;q) - axq(\tilde{H}_{k,1}(a;xq;q) + axq\tilde{H}_{k,0}(a;xq;q))}{1+xq}\\
    &=\frac{\tilde{H}_{k,2}(a;xq;q)}{1+xq}.
\end{align*}
In fact, again using (\ref{Hproperty1}) and (\ref{Hproperty2}) and setting $i=1$ in (\ref{Jtildetildedef}), we have that
\begin{align*}
    &\frac{\tilde{J}_{k,2}(a;x;q) + xq \tilde{J}_{k,0}(a;x;q)}{1+xq} \\
    &=\frac{\tilde{H}_{k,2}(a;xq;q) + axq\tilde{H}_{k,1}(a;xq;q) +xq(\tilde{H}_{k,0}(a;xq;q) + axq\tilde{H}_{k,-1}(a;xq;q))}{1+xq}\\
    &=\frac{\tilde{H}_{k,2}(a;xq;q) + axq\tilde{H}_{k,1}(a;xq;q) + xq( 0 - axq(xq)^{-1}\tilde{H}_{k,1}(a;xq;q))}{1+xq}\\
    &=\frac{\tilde{H}_{k,2}(a;xq;q)}{1+xq},
\end{align*}
so that we may use equation (\ref{Jtildetildedef}) to define $\tilde{\tilde{J}}_{k,1}(a;x;q)$ as well. 

We also extend (\ref{Jtildetildedef}) to the case $i=k$. Before doing so, we prove
\begin{lemma}
    \begin{equation}
        \tilde{J}_{k,k+1}(a;x;q) = \tilde{J}_{k,k-1}(a;x;q).
    \end{equation}
\end{lemma}
\begin{proof}
    We use (\ref{Hproperty1}) - (\ref{Hproperty3}). Indeed, we have
    \begin{align*}
    \tilde{J}_{k,k+1}(a;x;q)&=\frac{\tilde{H}_{k,0}(a;x;q)-\tilde{H}_{k,-2}(a;x;q)}{x^{-2}(1+x)}\\
    &=\frac{\tilde{H}_{k,0}(a;x;q)-(-x^{-2})\tilde{H}_{k,2}(a;x;q)}{x^{-2}(1+x)} \\
    &=\frac{\tilde{H}_{k,2}(a;x;q)}{(1+x)}\\
    &=\frac{\tilde{H}_{k,2}(a;x;q)-\tilde{H}_{k,0}(a;x;q)}{(1+x)}\\
    &=\tilde{J}_{k,k-1}(a;x;q).
\end{align*}
\end{proof}
Extending (\ref{Jtildetildedef}) to the case $i=k$ gives us
\begin{align}
    \tilde{\tilde{J}}_{k,k}(a;x;q) &= \frac{\tilde{J}_{k,k+1}(a;x;q) + (xq)\tilde{J}_{k,k-1}(a;x;q)}{1+xq}\nonumber\\
    &=\frac{\tilde{J}_{k,k-1}(a;x;q) + (xq)\tilde{J}_{k,k-1}(a;x;q)}{1+xq}\nonumber\\
    &= \label{JGhostDef3} \tilde{J}_{k,k-1}(a;x;q),
\end{align}
which, as we'll see below, corresponds exactly to the identification made in Remark \ref{Ghost1}.

We are now ready to provide a dictionary between the ghost series $\Tilde{G}_{\ell}$ and $\tilde{\tilde{J}}_{k,i}(a;x;q)$. We have the following formula for $\tilde{\tilde{J}}_{k,i}(a;x;q)$, where $1 \le i \le k$.
\begin{proposition}
    For $1 \le i \le k$,
\begin{align}\label{FullJtt}
    \tilde{\tilde{J}}_{k,i}(a;x;q)
    &=\frac{1}{1+xq}\sum_{n \ge 0} \frac{(-a)^n q^{kn^2-\binom{n}{2} + kn -(i+1)n}x^{(k-1)n}(-xq;q)_n(-1/a;q)_n(-axq^{n+2})_\infty}{(q^2;q^2)_n (xq^{n+1};q)_\infty}\nonumber\\
    &\hspace{4em}\cdot (1+xq^{2n+1})\left(1-x^iq^{(2n+1)i}+axq^{n+1}(1-x^{i-1}q^{(2n+1)(i-1)})\right).
\end{align}
\end{proposition}
\begin{proof}
   Using (\ref{FullJt}), we have
\begin{align*}
    &\tilde{J}_{k,i+1}(a;x;q) + xq\tilde{J}_{k,i-1}(a;x;q)=\\
    &=\sum_{n \ge 0} \frac{(-a)^n q^{kn^2-\binom{n}{2} + kn -(i+1)n }x^{(k-1)n}(-xq;q)_n(-1/a;q)_n(-axq^{n+2})_\infty}{(q^2;q^2)_n (xq^{n+1};q)_\infty}\nonumber\\
    &\hspace{1em}\cdot (1-x^{i+1}q^{(2n+1)(i+1)}+axq^{n+1}(1-x^{i}q^{(2n+1)i)}) \\
    &\hphantom{=}+xq\sum_{n \ge 0} \frac{(-a)^n q^{kn^2-\binom{n}{2} + kn -(i-1)n }x^{(k-1)n}(-xq;q)_n(-1/a;q)_n(-axq^{n+2})_\infty}{(q^2;q^2)_n (xq^{n+1};q)_\infty}\nonumber\\
    &\hspace{3em}\cdot (1-x^{i-1}q^{(2n+1)(i-1)}+axq^{n+1}(1-x^{i-2}q^{(2n+1)(i-2)})\\
    &=\sum_{n \ge 0} \frac{(-a)^n q^{kn^2-\binom{n}{2} + kn -(i+1)n }x^{(k-1)n}(-xq;q)_n(-1/a;q)_n(-axq^{n+2})_\infty}{(q^2;q^2)_n (xq^{n+1};q)_\infty}\nonumber\\
    &\hspace{1em}\cdot \Big(1-x^{i+1}q^{(2n+1)(i+1)}+axq^{n+1}(1-x^{i}q^{(2n+1)i)})\\
    &\hspace{3em}+xq^{2n+1}(1-x^{i-1}q^{(2n+1)(i-1)}+axq^{n+1}(1-x^{i-2}q^{(2n+1)(i-2)})\Big)\\
    &=\sum_{n \ge 0} \frac{(-a)^n q^{kn^2-\binom{n}{2} + kn -(i+1)n }x^{(k-1)n}(-xq;q)_n(-1/a;q)_n(-axq^{n+2})_\infty}{(q^2;q^2)_n (xq^{n+1};q)_\infty}\nonumber\\
    &\hspace{1em}\cdot \Big(1-x^{i+1}q^{(2n+1)(i+1)}+axq^{n+1}-ax^{i+1}q^{(2n+1)i+n+1}\\
    &\hspace{2.5em}+xq^{2n+1}-x^{i}q^{(2n+1)i}+ax^2q^{3n+2}-ax^{i}q^{(2n+1)(i-1)+n+1}\Big)\\
    &=\sum_{n \ge 0} \frac{(-a)^n q^{kn^2-\binom{n}{2} + kn -(i+1)n }x^{(k-1)n}(-xq;q)_n(-1/a;q)_n(-axq^{n+2})_\infty}{(q^2;q^2)_n (xq^{n+1};q)_\infty}\nonumber\\
    &\hspace{1em}\cdot \Big(1+xq^{2n+1}-x^{i}q^{(2n+1)i}-x^{i+1}q^{(2n+1)(i+1)}\\
    &\hspace{3em}+axq^{n+1}+ax^2q^{3n+2}-ax^{i}q^{(2n+1)(i-1)+n+1}-ax^{i+1}q^{(2n+1)i+n+1}\Big)\\
    &=\sum_{n \ge 0} \frac{(-a)^n q^{kn^2-\binom{n}{2} + kn -(i+1)n }x^{(k-1)n}(-xq;q)_n(-1/a;q)_n(-axq^{n+2})_\infty}{(q^2;q^2)_n (xq^{n+1};q)_\infty}\nonumber\\
    &\hspace{1em}\cdot (1+xq^{2n+1})\left(1-x^iq^{(2n+1)i}+axq^{n+1}(1-x^{i-1}q^{(2n+1)(i-1)})\right).
\end{align*} 
The claim now follows from (\ref{Jtildetildedef}).
\end{proof}

Using a proof similar to Proposition \ref{dictionary}, and in light of Remark \ref{Ghost1}, we now immediately have:
\begin{proposition} 
When $j=0$ and $1 \le i \le k$ or when $j>0$ and  $2 \le i \le k $, we have
    \begin{equation}
        \tilde{G}_{(k-1)j+i} = \tilde{\tilde{J}}_{k,k-i+1}(1/q, q^{2j},q^2).
    \end{equation}
\end{proposition}
\noindent We note here that this proposition can now serve as a definition for $\tilde{G}_{(k-1)j+1}$, which has only been defined for $j=0$ in Remark \ref{Ghost1}.

We now explore the series $\tilde{\tilde{J}}_{k,i}(a;x;q)$ in more detail, and note its similarities to the series $\tilde{J}(a;x;q)$.
Using (\ref{Jtildetildedef}), we have, for $1 \le i \le k-1$, that 
\begin{align}\label{JtildetildeExpanded}
    \tilde{\tilde{J}}_{k,i}(a;x;q) &= \frac{\tilde{J}_{k,i+1}(a;x;q) + xq \tilde{J}_{k,i-1}(a;x;q)}{1+xq}\nonumber \\
    &=\frac{\tilde{H}_{k,i+1}(a;xq;q) + axq\tilde{H}_{k,i}(a;xq;q) +xq\left(\tilde{H}_{k,i-1}(a;xq;q) + axq\tilde{H}_{k,i-2}(a;xq;q)\right)}{1+xq}\nonumber\\
    &= \frac{\tilde{H}_{k,i+1}(a;xq;q) + xq\tilde{H}_{k,i-1}(a;xq;q)}{1+xq} + axq\frac{\tilde{H}_{k,i}(a;xq;q) + xq\tilde{H}_{k,i-2}(a;xq;q)}{1+xq}.
\end{align}
Let 
\begin{align}
    \tilde{\tilde{H}}_{k,i}(a,x,q):= \frac{\tilde{H}_{k,i+1}(a;x;q) + x\tilde{H}_{k,i-1}(a;x;q)}{1+x},
\end{align} so that (\ref{JtildetildeExpanded}) can be written as
\begin{align}
    \tilde{\tilde{J}}_{k,i}(a;x;q) = \tilde{\tilde{H}}_{k,i}(a;xq;q) + axq\tilde{\tilde{H}}_{k,i-1}(a;xq;q).
\end{align}
We propose the following definition for $\tilde{\tilde{J}}_{k,i}(a;x;q)$, in place of (\ref{Jtildetildedef}), for $1 \le i \le k$:
\begin{defi}\label{JttDef}
    For $k \ge 2$, we define
    \begin{equation}
        \tilde{\tilde{J}}_{k,i}(a;x;q) = \tilde{\tilde{H}}_{k,i}(a;xq;q) + axq\tilde{\tilde{H}}_{k,i-1}(a;xq;q).
    \end{equation}
\end{defi}
We also note that we have the following properties, which follow immediately from (\ref{Hproperty1}) and (\ref{Hproperty2}).
\begin{lemma}
    \begin{equation}\label{Htt1}
        \tilde{\tilde{H}}_{k,0}(a;x;q) = 0
    \end{equation}
    \begin{equation}\label{Htt2}
        \tilde{\tilde{H}}_{k,-i}(a;x;q) = -x^{-i}\tilde{\tilde{H}}_{k,i}(a;x;q).
    \end{equation}
\end{lemma}
\begin{proof}
    We first show (\ref{Htt2}). We have
    \begin{align*}
        \tilde{\tilde{H}}_{k,-i}(a;x;q) &= \frac{\tilde{H}_{k,-i+1}(a;x;q) + x\tilde{H}_{k,-i-1}(a;x;q)}{1+x}\\
        &=\frac{-x^{-i+1}\tilde{H}_{k,i-1}(a;x;q) -x^{-i}\tilde{H}_{k,i+1}(a;x;q)}{1+x}\\
        &=-x^{-i}\frac{\tilde{H}_{k,i+1} + x\tilde{H}_{k,i-1}(a;x;q)}{1+x}\\
        &=-x^{-i}\tilde{\tilde{H}}_{k,i}(a;x;q).
    \end{align*}
    Now, (\ref{Htt1}) follows from (\ref{Htt2}) with $i=0$. 
\end{proof}

Lastly, we conclude with a combinatorial interpretation of the $\tilde{\tilde{J}}_{k,i}(a;x;q)$. The proof is nearly identical to the proof of Theorem \ref{GhostCombinatorics} and follows from Theorem 1.1 in \cite{CoLoMa}.  For an overpartition $\lambda$ of a nonnegative integer, $n$, let $V_{\lambda}(\ell)$ denote the number of overlined parts in $\lambda$ which do not exceed $\ell$. 
\begin{theo}
    For $1 \le i \le k$, let $\tilde{c}_{k,i}(j,m,n)$ denote the number of overpartitions $\lambda$ of $n$ with $m$ parts and $j$ overlined parts satisfying the conditions: 
    \begin{enumerate}\item $f_1(\lambda) + f_{\overline{1}}(\lambda) \le i -1$,
    \item $f_{\ell}(\lambda) + f_{\ell+1}(\lambda) + f_{\overline{\ell + 1}}(\lambda) \le k-1$,
    \item if $f_{\ell}(\lambda) + f_{\ell+1}(\lambda) + f_{\overline{\ell + 1}}(\lambda) = k-1$, then 
    \begin{equation*}\ell f_{\ell}(\lambda) + (\ell + 1)(f_{\ell+1}(\lambda) + f_{\overline{\ell+1}}(\lambda)) \equiv i + V_{\lambda}(\ell) \bmod 2.\end{equation*}
    \end{enumerate}
    Then
    \begin{equation}
        \tilde{\tilde{J}}_{k,i}(a;x;q) = \sum_{j,m,n \ge 0}\tilde{c}_{k,i}(j,m,n)a^j x^m q^n.
    \end{equation}
\end{theo}
\begin{proof}
From Theorem 1.1 in \cite{CoLoMa}, we have that 
    \begin{equation}
        \tilde{J}_{k,i}(a;x;q) = \sum_{j,m,n \ge 0}c_{k,i}(j,m,n)a^j x^m q^n.
    \end{equation}
where $c_{k,i}(j,m,n)$ denotes the number of overpartitions $\lambda$ of $n$ with $m$ parts and $j$ overlined parts satisfying the conditions
    \begin{enumerate}\item $f_1(\lambda) + f_{\overline{1}}(\lambda) \le i -1$,
    \item $f_{\ell}(\lambda) + f_{\ell+1}(\lambda) + f_{\overline{\ell + 1}}(\lambda) \le k-1$,
    \item if $f_{\ell}(\lambda) + f_{\ell+1}(\lambda) + f_{\overline{\ell + 1}}(\lambda) = k-1$, then 
    \begin{equation*}\ell f_{\ell}(\lambda) + (\ell + 1)(f_{\ell+1}(\lambda) + f_{\overline{\ell+1}}(\lambda)) \equiv i +1 + V_{\lambda}(\ell) \bmod 2.\end{equation*}
    \end{enumerate}
From this, we have that $\tilde{J}_{k,i}(a;xq;q)$
is the generating function for overpartitions $\lambda$ into $m$ parts with $j$ overlined parts satisfying
\begin{enumerate}\item $f_2(\lambda) + f_{\overline{2}}(\lambda) \le i -1$,
    \item $f_{\ell}(\lambda) + f_{\ell+1}(\lambda) + f_{\overline{\ell + 1}}(\lambda) \le k-1$,
    \item if $f_{\ell+1}(\lambda) + f_{\ell+2}(\lambda) + f_{\overline{\ell + 2}}(\lambda) = k-1$, then 
    \begin{equation}\label{Con3xq}\ell f_{\ell+1}(\lambda) + (\ell + 1)(f_{\ell+2}(\lambda) + f_{\overline{\ell+2}}(\lambda)) \equiv i +1 + V_{\lambda}(\ell+1) \bmod 2.\end{equation}
    \end{enumerate}
Making the substitution $\ell \mapsto \ell-1$ and using $f_{\ell}(\lambda) + f_{\ell+1}(\lambda) + f_{\overline{\ell + 1}}(\lambda) = k-1$, equation (\ref{Con3xq}) can be rewritten as 
\begin{equation*}\ell f_{\ell}(\lambda) + (\ell + 1)(f_{\ell+1}(\lambda) + f_{\overline{\ell+1}}(\lambda)) \equiv k+i + V_{\lambda}(\ell) \bmod 2.\end{equation*}
Thus we have that  $\tilde{J}_{k,i}(a;xq;q)$
is the generating function for overpartitions $\lambda$ into $m$ parts with $j$ overlined parts satisfying
\begin{enumerate}\item $f_2(\lambda) + f_{\overline{2}}(\lambda) \le i -1$,
    \item $f_{\ell}(\lambda) + f_{\ell+1}(\lambda) + f_{\overline{\ell + 1}}(\lambda) \le k-1$,
    \item if $f_{\ell}(\lambda) + f_{\ell+1}(\lambda) + f_{\overline{\ell + 1}}(\lambda) = k-1$, then 
    \begin{equation*}\ell f_{\ell}(\lambda) + (\ell + 1)(f_{\ell+1}(\lambda) + f_{\overline{\ell+1}}(\lambda)) \equiv k+i + V_{\lambda}(\ell) \bmod 2.\end{equation*}
    \end{enumerate}
Now, using this fact along with (\ref{JghostDef1}), (\ref{JghostDef2}), (\ref{JGhostDef3}) and an argument similar to Theorem \ref{GhostCombinatorics} gives the desired result.
\end{proof}

\begin{rema}
    Using the ideas in this section, it is now clear how to set up a similar motivated proof for any partition or overpartition identity arising from the study of the series $\tilde{J}_{k,i}(a;x;q)$. For example, as in Corollaries 1.2-1.4 of \cite{CoLoMa}. Although the vertex-algebraic interpretation of the ghost series is not clear, we expect that the ghost series $\tilde{\tilde{J}}_{k,i}(a;x;q)$ will have interesting combinatorial properties related to the properties of the series $\tilde{J}_{k,i}(a;x;q)$, which are the subject of \cite{CoLoMa}. 
\end{rema}

\vspace{.3in}

\noindent {\small \sc Department of Mathematics, University of Virginia,
Charlottesville, VA 22903} \\ 
{\em E--mail address}: 
\texttt{bhj3pr@virginia.edu} \\

\noindent {\small \sc Department of Mathematics, University of Central Florida,
Orlando, FL 32816} \\
{\em E--mail address}:
\texttt{sa318800@ucf.edu} \\

\noindent {\small \sc Department of Mathematics, Computer Science, and Statistics, Ursinus College,
Collegeville, PA 19426} \\
{\em E--mail address}:
\texttt{csadowski@ursinus.edu} \\

\noindent {\small \sc Department of Mathematics and Computer Science, Dickinson College,
Carlisle, PA 17013} \\
{\em E--mail address}:
\texttt{shambaue@dickinson.edu} \\

\


\begin{thebibliography}{CKLMQRS}

\bibitem[A1]{A1} G. E. Andrews, An analytic proof of the
Rogers-Ramanujan-Gordon identities, {\em Amer. J. Math.} {\bf 88}
(1966), 844-846.

\bibitem[A2]{A2} G. E. Andrews, 
A generalization of the G\"ollnitz-Gordon partition theorems,
{\em Proc. Amer. Math. Soc.} {\bf 18} (1967), 945-952.

\bibitem[A3]{A3} G. E. Andrews, {\em The Theory of Partitions},
Encyclopedia of Mathematics and Its Applications, Vol. 2,
Addison-Wesley, 1976.

\bibitem[A4]{A4} G. E. Andrews, Letter to the editor, {\em American
Math. Monthly} {\bf 97} (1990), 215.


\bibitem[AB]{AB} G. E. Andrews and R. J. Baxter, A motivated proof of
the Rogers-Ramanujan identities, {\em American Math. Monthly} {\bf 96}
(1989), 401-409.

\bibitem[BKRS]{BKRS} K. Baker, S. Kanade, M.C. Russell, and C. Sadowski, Principal subspsaces of basic modules for twisted affine Lie algebras, $q$-series multisums, and Nandi's identities, {\em Algebr. Comb.}, to appear,
{\tt arXiv:2208.14581 [math.CO]}.



\bibitem[B]{B} R. J. Baxter, Hard hexagons: exact solution, {\em
J. Physics A} {\bf 13} (1980), L61-L70.

\bibitem[Br]{Br} D. Bressoud, Analytic and combinatorial generalizations
of the Rogers-Ramanujan identities, {\em Mem. Amer. Math. Soc.} 227 (1980).
 



\bibitem[CalLM]{CalLM} C. Calinescu, J. Lepowsky and A. Milas,
Vertex-algebraic structure of principal subspaces of standard
$A_2^{(2)}$-modules, I, {\em Internat. J. of Math.} {\bf 25} (2014),
1450063.

\bibitem[CLM1]{CLM1} S. Capparelli, J. Lepowsky and A. Milas, The
Rogers-Ramanujan recursion and intertwining operators, {\em Comm. in
Contemp. Math.} {\bf 5} (2003), 947-966.

\bibitem[CLM2]{CLM2} S. Capparelli, J. Lepowsky and A. Milas, The
Rogers-Selberg recursions, the Gordon-Andrews identities and
intertwining operators, {\em The Ramanujan Journal} {\bf 12} (2006),
377-395.

\bibitem[CMPP]{CMPP} S. Capparelli, A. Meurman, A. Primc, and M. Primc,
New partition identities from $C_{\ell}^{(1)}$-modules,
{\em Glas. Mat. Ser. III}, 57(77)(2):161-184, 2022.

\bibitem[CoLoMa]{CoLoMa} S. Corteel, J. Lovejoy, and O. Mallet, An
extension to overpartitions of the Rogers-Ramanujan identities
for even moduli, {\em J. Number Theory} 128 (2008), no. 6, 1602-1621.

\bibitem[CKLMQRS]{CKLMQRS}B. Coulson, S. Kanade, J. Lepowsky, R. McRae, F. Qi, M.C. Russell,
C. Sadowski, A motivated proof of the Göllnitz--Gordon--Andrews identities, {\em Ramanujan J.}
{\bf 42} (2017), 97-129.

\bibitem[DL]{DL} C. Dong and J. Lepowsky, {\em Generalized Vertex
Algebras and Relative Vertex Operators}, Progress in Mathematics,
Vol. 112, Birkh\"auser, Boston, 1993.

\bibitem[DK]{DK} J. Dousse and I. Konan, Characters of level $1$ standard
modules of $C_n^{(1)}$ as generating functions for generalised partitions, 
arXiv:2212.12728 [math.CO], 2022.


\bibitem[G\"ol]{Gol} H. G\"ollnitz, 
Partitionen mit Differenzenbedingungen (German),
{\em J. Reine Angew. Math.} {\bf 225} (1967), 154-190.

\bibitem[G]{G} B. Gordon, A combinatorial generalization of the
Rogers-Ramanujan identities, {\em Amer. J. Math.} {\bf 83} (1961),
393-399.


\bibitem[HJZ1]{HJZ1}
He, T. Y., Ji, K. Q., and Zhao, A. X. H., Overpartitions and Bressoud’s Conjecture, I, {\em Adv. Math.}, 404 (2022) 108449.


\bibitem[HJZ2]{HJZ2}
He, T. Y., Ji, K. Q., and Zhao, A. X. H., Overpartitions and Bressoud’s Conjecture, II, 
arXiv:2001.00162 [math.CO], 2020.


\bibitem[HWZ]{HWZ} He, T. Y., Wang, A. Y. F., and Zhao, A. X. H., The Bressoud-G{\"o}llnitz-Gordon Theorem for Overpartitions of Even Moduli, {\em Taiwanese Journal of Mathematics}, 21(6) (2017).

\bibitem[HZ]{HZ} He, T.Y., Zhao, A. X. H., New companions to the generalizations of the Göllnitz–Gordon identities. {\em Ramanujan J.} 61, 1077–1120 (2023). https://doi.org/10.1007/s11139-023-00715-3



\bibitem[Hu]{Hu} C. Husu, Extensions of the Jacobi identity for vertex
operators, and standard $A_1^{(1)}$-modules, {\em Mem. Amer. Math. Soc.}
{\bf 106} (1993), no. 507.


\bibitem[K]{K} S. Kanade, Lepowsky–Wilson 
$Z$-algebras and Rogers–Ramanujan-type identities: Recent advances, https://cs.du.edu/~shakanad/assets/RRLepowskyWilson.pdf

\bibitem[KLRS]{KLRS} S. Kanade, J. Lepowsky, M. C. Russell, and A. V. Sills, Ghost series and a motivated proof of the Andrews-Bressoud identities, {\em J. Combin. Theory Ser. A} {\bf 146} (2017), 33-62.

\bibitem[L]{L} J. Lepowsky, Some developments in vertex operator algebra theory, 
old
and new, in: {\it Lie Algebras, Vertex Operator Algebras and Their
Applications}, ed. by Y.-Z. Huang and K. C. Misra, Contemp. Math.,
Vol. 442, American Math. Soc., 2007, 355-387.


\bibitem[LM]{LM} J. Lepowsky and S. Milne, Lie algebraic approaches to
classical partition identities, {\em Adv. in Math.} { \bf 29} (1978),
15-59.

\bibitem[LW1]{LW1} J. Lepowsky and R. L. Wilson, Construction of the
affine Lie algebra $A_1^{(1)}$, {\em Comm. Math. Phys.}  {\bf 62}
(1978) 43-53.

\bibitem[LW2]{LW2} J. Lepowsky and R. L. Wilson, A new family of
algebras underlying the Rogers-Ramanujan identities, {\em
Proc. Nat. Acad. Sci. USA} {\bf 78} (1981), 7254-7258.

\bibitem[LW3]{LW3} J. Lepowsky and R. L. Wilson, The structure of
standard modules, I: Universal algebras and the Rogers-Ramanujan
identities, {\em Invent. Math.} {\bf 77} (1984), 199-290.

\bibitem[LW4]{LW4} J. Lepowsky and R. L. Wilson, The structure of
standard modules, II: The case $A_1^{(1)}$, principal gradation, {\em
Invent. Math.} {\bf 79} (1985), 417-442.

\bibitem[LZ]{LZ} J. Lepowsky and M. Zhu, 
A motivated proof of Gordon's identities,
{\em Ramanujan J.} {\bf 29} (2012), 199-211. 

\bibitem[MP]{MP} A. Meurman and M. Primc, 
Annihilating ideals of standard modules of $\mathfrak{sl}(2,\mathbb{C})^{\sim}$ 
and combinatorial identities,
{\em Adv. in Math.} {\bf 64} (1987), 177-240.

\bibitem[PSW]{PSW} M. Penn, C. Sadowski, and G. Webb, Principal subspaces of twisted modules for certain lattice vertex operator
algebras, {\em Internat. J. Math.}, 30(10):1950048, 47, 2019.

\bibitem[RR]{RR} S. Ramanujan and L. J. Rogers, Proof of certain
identities in combinatory analysis, {\em Proc. Cambridge Phil. Soc.}
{\bf 19} (1919), 211-214.

\bibitem[R]{R} R. M. Robinson, Letter to the editor, {\em American
Math. Monthly} {\bf 97} (1990), 214-215.

\bibitem[Ru]{Ru} M. C. Russell, Companions to the Andrews-Gordon and Andrews-Bressoud identities, and recent conjectures of Capparelli, Meurman, Primc, and Primc, 
arXiv:2306.16251 [math.CO], 2023.

\end{thebibliography}
\end{document}